\documentclass{amsart}
\usepackage{amsmath}
\usepackage{graphicx}
\usepackage{amsfonts}
\usepackage{amsthm}
\usepackage{centernot}
\usepackage{setspace}
\usepackage{mathtools}
\usepackage{amssymb}
\usepackage{bm}
\usepackage{hyperref}
\usepackage[capitalize]{cleveref}
\usepackage{mathrsfs}
\usepackage{tikz}
\usepackage{tikz-cd}
\usepackage{breqn}
\usepackage[strict]{changepage}
\usepackage{breqn}
\newtheorem{thm}{Theorem}[section]

\newtheorem{prop}[thm]{Proposition}
\newtheorem{lem}[thm]{Lemma}

\theoremstyle{definition}
\newtheorem{defn}[thm]{Definition}

\theoremstyle{remark}

\makeatletter
\let\c@equation\c@thm
\makeatother
\numberwithin{equation}{section}

\newcommand{\ov}[1]{\,\overline{#1}}

\DeclareMathOperator{\aut}{Aut}

\DeclareMathOperator{\im}{im}

\DeclareMathOperator{\Iw}{Iw}
\DeclareMathOperator{\ad}{ad}
\DeclareMathOperator{\Art}{Art}

\DeclareMathOperator{\Ann}{Ann}
\DeclareMathOperator{\Sp}{Sp}

\DeclareMathOperator{\GL}{GL}

\DeclareMathOperator{\Hom}{Hom}

\DeclareMathOperator{\ind}{ind}
\DeclareMathOperator{\res}{Res}

\DeclareMathOperator{\End}{End}

\DeclareMathOperator{\Tr}{Tr}

\DeclareMathOperator{\diag}{diag}

\DeclareMathOperator{\Frob}{Frob}
\DeclareMathOperator{\rec}{rec}

\newcommand{\Z}{\mathbf{Z}}
\newcommand{\Q}{\mathbf{Q}}

\newcommand{\C}{\mathbf{C}}
\newcommand{\D}{\mathcal{D}}
\newcommand{\OO}{\mathcal{O}}

\newcommand{\m}{\mathfrak{m}}
\newcommand{\n}{\mathfrak{n}}
\newcommand{\p}{\mathfrak{p}}

\newcommand{\F}{\mathbf{F}}

\newcommand{\A}{\mathbf{A}}

\newcommand{\V}{\mathcal{V}}
\newcommand{\HH}{\mathcal{H}}
\newcommand{\TT}{\mathbf{T}}
\newcommand{\DD}{\mathbf{D}}
\newcommand{\CC}{\mathcal{C}}
\newcommand{\isom}{\xrightarrow{\sim}}
\newcommand{\Sat}{\mathcal{S}}
\newcommand{\til}{\widetilde}

\title{Automorphy lifting with adequate image}
\author{Konstantin Miagkov, Jack A. Thorne} 
\date{}
\begin{document}

\begin{abstract}
Let $F$ be a CM number field. We generalize existing automorphy lifting theorems for regular residually irreducible $p$-adic Galois representations over $F$ by relaxing the big image assumption on the residual representation.
\end{abstract}

\maketitle

\tableofcontents

\section{Introduction}
This paper closely builds on \cite{10author}, which proves modularity lifting theorems for regular $n$-dimensional Galois representations over a CM number field $F$ without any self-duality condition. In this paper, we generalize the main results of \cite{10author} to relax the big image assumption on the residual representation from `enormous image' to `adequate image'. Following \cite{Thorne12} we define `adequate image':
\begin{defn}
Let $k$ be a finite field of characteristic $p$ such that $p \nmid n$, and let $G \subset \GL_n(k)$ be a subgroup which acts absolutely irreducibly on $V = k^n$. We suppose that $k$ is large enough to contain all eigenvalues of all elements of $G$. If $g \in G$ and $\alpha \in k$ is an eigenvalue $g$, we write $e_{g, \alpha} : V \to V$ for the $g$-equivariant projection to the generalized $\alpha$-eigenspace. We say that $G$ is \textbf{adequate} if the following conditions are satisfied:
\begin{enumerate}
\item $H^0(G, \ad^0 V) = 0$.
\item $H^1(G, k) = 0.$
\item $H^1(G, \ad^0 V) = 0$.
\item For every irreducible $k[G]$-submodule $W \subset \ad^0 V$ there exists an element $g \in G$ with an eigenvalue $\alpha$ such that $\text{tr}~e_{g, \alpha}W \neq 0$.
\end{enumerate}
\end{defn}

Our main theorems are as follows:
\begin{thm}\label{thm:main1}
Let $F$ be an imaginary $CM$ or totally real field, let $c \in \aut(F)$ be complex conjugation, and let $p$ be a prime. Suppose given a continuous representation $\rho : G_F \to \GL_n(\ov{\Q}_p)$ satisfying the following conditions:
\begin{enumerate}
\item $\rho$ is unramified almost everywhere.
\item For each place $v \mid p$ of $F$, the representation $\rho|_{G_{F_v}}$ is crystalline. The prime $p$ is unramified in $F$.
\item $\ov{\rho}$ is absolutely irreducible and decomposed generic. The image of $\ov{\rho}|_{G_{F(\zeta_p)}}$ is adequate.
\item There exists $\sigma \in G_F - G_{F(\zeta_p)}$ such that $\ov{\rho}(\sigma)$ is a scalar. We have $p > n^2$.
\item There exists a cuspidal automorphic representation $\pi$ of $\GL_n(\mathbf{A}_F)$ satisfying the following conditions:
    \begin{enumerate}
    \item $\pi$ is regular algebraic of weight $\lambda$, this weight satisfying \[\lambda_{\tau,1} + \lambda_{\tau c,1} - \lambda_{\tau,n} - \lambda_{\tau c,n} < p-2n\] for all $\tau$.
    \item There exists an isomorphism $\iota : \ov{\Q}_p \to \C$ such that $\ov{\rho} \cong \ov{r_\iota(\pi)}$ and the Hodge-Tate weights of $\rho$ satisfy the formula for each $\tau : F \xhookrightarrow{} \ov{\Q}_p$:
    \[HT_{\tau}(\rho) = \{\lambda_{\iota \tau,1} + n - 1, \lambda_{\iota \tau,2} + n - 2, \ldots, \lambda_{\iota \tau,n}\}.\]
    \item If $v \mid p$ is a place of $F$, then $\pi_v$ is unramified.
    \end{enumerate}
\end{enumerate}
Then $\rho$ is automorphic: there exists a cuspidal automorphic representation $\Pi$ of $\GL_n(\mathbf{A}_F)$ of weight $\lambda$ such that $\rho \cong r_\iota(\Pi)$ Moreover, if $v$ is a finite place of $F$ and either $v \mid p$ or both $\rho$ and $\pi$ are unramified at $v$, then $\Pi_v$ is unramified.
\end{thm}

\begin{thm}\label{thm:main2}
Let $F$ be an imaginary $CM$ or totally real field, let $c \in \aut(F)$ be complex conjugation, and let $p$ be a prime. Suppose given a continuous representation $\rho : G_F \to \GL_n(\ov{\Q}_p)$ satisfying the following conditions:
\begin{enumerate}
\item $\rho$ is unramified almost everywhere.
\item Let $\Z_+^n = \{(\lambda_1, \ldots, \lambda_n) \in \Z^n \mid \lambda_1 \geq \ldots \geq \lambda_n\}$. For each place $v \mid p$ of $F$, the representation $\rho|_{G_{F_v}}$ is potentially semistable, ordinary with regular Hodge-Tate weights. In other words, there exists a weight $\lambda \in (\Z_+^n)^{\Hom(F, \ov{\Q}_p)}$ such that for each place $v \mid p$, there is an isomorphism 
\[\rho|_{G_{F_v}} \sim 
\begin{pmatrix}
    \psi_{v,1} & * & * & * \\
    0 & \psi_{v,2} & * & * \\
    \vdots & \ddots & \ddots & * \\
    0 & \dots & 0 & \psi_{v,n}
\end{pmatrix},
\]
where for each $i = 1,\ldots,n$ the character $\psi_{v,i} : G_{F_v} \to \ov{\Q}_p^\times$ agrees with the character 
\[\sigma \in I_{F_v} \mapsto \prod_{\tau \in \Hom(F_v, \ov{\Q}_p)} \tau(\Art_{F_v}^{-1}(\sigma))^{-(\lambda_{\tau, n-i+1} +i-1)}\]
on an open subgroup of the inertia group $I_{F_v}$.
\item $\ov{\rho}$ is absolutely irreducible and decomposed generic. The image of $\ov{\rho}|_{G_{F(\zeta_p)}}$ is adequate.
\item There exists $\sigma \in G_F - G_F(\zeta_p)$ such that $\ov{\rho}(\sigma)$ is a scalar. We have $p > n$.
\item There exists a cuspidal automorphic representation $\pi$ of $\GL_n(\mathbf{A}_F)$ and an isomorphism $\iota : \ov{\Q}_p \to \C$ such that $\pi$ is $\iota$-ordinary and $\ov{\rho} \cong \ov{r_\iota(\pi)}$.
\end{enumerate}
Then $\rho$ is ordinarily automorphic of weight $\iota\lambda$: there exists a $\iota$-ordinary cuspidal automorphic representation $\Pi$ of $\GL_n(\mathbf{A}_F)$ of weight $\iota\lambda$ such that $\rho \cong r_\iota(\Pi)$ Moreover, if $v \nmid p$ is a finite place of $F$ and both $\rho$ and $\pi$ are unramified at $v$, then $\Pi_v$ is unramified.
\end{thm}

The theorems above are very similar to \cite[Theorem~6.1.1]{10author} and \cite[Theorem~6.1.2]{10author} respectively. The only difference is replacing the \textit{enormous} condition on image of $\ov{\rho}|_{G_{F(\zeta_p)}}$ with \textit{adequate}. This is a useful improvement, particularly in light of \cite{adequate} which proves that when $p > 2(n+1)$ adequacy is equivalent to absolute irreducibility. This makes it a condition easy to work with in the context of automorphy of compatible systems, which we hope would help generalize \cite{BLGGT} to the context of \cite{10author} and this paper. 

We now give a brief overview of the argument. The main change in comparison to \cite{10author} is the usage of parahoric-level subgroups at Taylor-Wiles primes instead of Iwahori-level, the idea first introduced to relax the big image assumption in the setting of automorphy lifting theorems to `adequate' in \cite{Thorne12}. To make the argument work in the parahoric setting, we need to analyze the representations of $\GL_n(F_v)$ with fixed vectors under various parahoric subgroups and their interactions with the local Langlands correspondence. A notable difficulty in comparison to \cite{Thorne12} is that we can no longer restrict to working with generic local representations, since they arise as components of cuspidal automorphic representations of unitary groups instead of $\GL_n$. The local computations allow us to prove the necessary local-global compatibility results for Galois representations landing in Hecke algebras acting on cohomology of locally symmetric spaces with parahoric level. Another novel component is a proof of a "growth of the space of cusp forms"-type result when adding Taylor-Wiles primes with parahoric level, which requires an investigation of representations of $\GL_n(F_v)$ over fields of finite characteristic. 

\subsection{Acknowledgments}

K.M. would like to thank his advisor Richard Taylor for many helpful conversations. J.T.'s work received funding from the European Research Council (ERC) under the European Union's Horizon 2020 research and innovation programme (grant agreement No 714405).

\subsection{Notation}

We write $GL_n$ for the usual general linear group (viewed as a reductive group scheme over $Z$) and $T_n \subset  B_n \subset GL_n$ for its subgroups of diagonal and of upper triangular matrices, respectively. We identify $X^*(T)$ with $\Z^n$ in the usual way, and write $\Z^n_+ \subset \Z^n$ for the subset of $B_n$-dominant weights. 
If $R$ is a local ring, we write $\m_R$ for the maximal ideal of $R$. If $\Gamma$ is a profinite group and $\rho : \Gamma \to GL_n(\ov{\Q}_p)$ is a continuous homomorphism, then we will let $\ov{\rho} : \Gamma \to GL_n(\ov{\F}_p)$ denote the \textit{semi-simplification} of its reduction, which is well defined up to conjugacy (by the Brauer-Nesbitt theorem). If $M$ is a topological abelian group with a continuous action of $\Gamma$, then by $H^i(\Gamma, M)$ we shall mean the continuous cohomology.

If $G$ is a locally profinite group, $U \subset G$ is an open compact subgroup and $R$ is a commutative ring, then we write $\HH_R(U,G)$ for the algebra of compactly supported, $U$-biinvariant functions $f : G \to R$, with multiplication given by convolution with respect to the Haar measure on $G$ which gives $U$ volume 1. If $X \subset G$ is a compact $U$-biinvariant subset, then we write $[X]$ for the characteristic function of $X$, an element of $\HH_R(U, G)$. When $R$ is omitted from the notation we take $R = \Z$. We write $\iota_{\HH}$ for the anti-involution given by $\iota_{\HH}(f)(g) = f(g^{-1})$.

If $F$ is a perfect field, we let $\ov{F}$ denote an algebraic closure of $F$ and $G_F$ the absolute Galois group $\text{Gal}(\ov{F}/F)$. We will use $\zeta_n$ to denote a primitive $n$-th root of unity when it exists. Let $\epsilon_l$ denote the $l$-adic cyclotomic character. 

We will let $\rec_K$ be the local Langlands correspondence of \cite{HT01}, so that if $\pi$ is an irreducible complex admissible representation of $GL_n(K)$, then $\rec_K(\pi)$ is a Frobenius semisimple Weil-Deligne representation of the Weil group $W_K$. 
If $K$ is a finite extension of $\Q_p$ for some $p$, we write $K^{nr}$ for its maximal
unramified extension; $I_K$ for the inertia subgroup of $G_K$ ; $\Frob_K \in G_K / I_K$ for
the geometric Frobenius; and $W_K$ for the Weil group. We will write $\text{Art}_K : K^\times \isom W_K^{\text{ab}}$ for the Artin map normalized to send uniformizers to geometric Frobenius elements.

We will write $\rec$ for $\rec_K$ when the choice of $K$ is clear. We write $\rec^T_K$ for the normalization of the local Langlands correspondence as defined in e.g. \cite[\S 2.1]{CT14}; it is defined on irreducible admissible representations of $GL_n(K)$ defined over any field which is abstractly isomorphic to $\C$ (e.g. $\ov{\Q}_l$). If $(r,N)$ is a Weil-Deligne representation of $W_K$, we will write $(r,N)^{F-ss}$ for its Frobenius semisimplification. If $\rho$ is a continuous representation of $G_K$ over $\ov{\Q}_l$ with $l \neq p$ then we will write $WD(\rho)$ for the corresponding Weil-Deligne representation of $W_K$. By a Steinberg representation of $GL_n(K)$ we will mean a representation $Sp_n(\psi)$ (in the notation of section 1.3 of \cite{HT01}) where $\psi$ is an unramified character of $K^\times$.

If $G$ is a reductive group over $K$ and $P$ is a parabolic subgroup with
unipotent radical $N$ and Levi component $L$, and if $\pi$ is a smooth representation
of $L(K)$, then we define $\text{Ind}_{P(K)}^{G(K)}\pi$ to be the set of locally constant functions 
$f : G(K) \to \pi$ such that $f(hg) = \pi(hN(K))f(g)$ for all $h \in P(K)$ and $g \in G(K)$. It is a smooth representation of $G(K)$ where $(g_1f)(g_2) = f(g_2g_1)$. This is sometimes referred to as `un-normalized' induction. We let $\delta_P$ denote the determinant of the action of $L$ on $Lie_N$. Then we define the
`normalized' induction $\ind_{P(K)}^{G(K)}\pi$ to be $\text{Ind}_{P(K)}^{G(K)}(\pi \otimes \lvert \delta_P \rvert_K^{1/2})$. 

If $F$ is a CM number field and $\pi$ is an automorphic representation of $\GL_n(\A_F)$, we say that $\pi$ is regular algebraic if $\pi_\infty$ has the same infinitesimal character as an irreducible algebraic representation $W$ of $(\res_{F/\Q}\GL_n)_\C$. If $W^\vee$ has highest weight $\lambda \in (\Z^n_+)^{\Hom(F, \C)}$, then we say $\pi$ has weight $\lambda$. 

If $P(X) \in A[X]$ is a polynomial of degree $n$ over any ring $A$ such that $P(0) \in A^\times$, we write $P^\vee(X)$ for $P(0)^{-1}X^nP(X^{-1})$. For two polynomials $P, Q \in A[X]$ we write $\res(P, Q)$ to denote their resultant.

Given a Galois representation $\rho : G_{F, S} \to \GL_n(A)$ we will write $\rho^\perp \coloneqq \rho^{c, \vee} \otimes \epsilon^{1-2n}$, and given a $G_{F, S}$-group determinant $D$ we will denote by $D^\perp$ the corresponding dual. 
\section{Representation theory of $\GL_n(F_v)$ in characteristic $p$}\label{sec:reps}

Let $p$ be a rational prime, and $k$ be a finite field of characteristic $p$. Let $F/\Q$ be a finite extension, and let $x$ be
a prime in $F$ with residue field $k_x$ of order $q$ satisfying $q \equiv 1 \pmod p$ and the corresponding ring of integers $\OO_x = \OO_{F_x}$. Set $G_x = Gal(\ov{F}_x/F_x)$. Also set $G = GL_n$ with $p > n$, and let $T \subset B \subset G$ be the maximal torus and the corresponding Borel, and $U \subset G$ be the unipotent subgroup. Let $K^1(x) \subset G(\OO_x)$ be the full congruence subgroup. We also let $\Iw, \Iw^1(x)  \subset G(\OO_x)$ be the Iwahori and the Iwahori-1 respectively,
and set $\Iw_1 \subset \Iw^p \subset \Iw$ be the subgroup such that $[\Iw^p : \Iw_1]$ has order prime to $p$ and 
$[\Iw : \Iw^p]$ has $p$-power order. Let $\p(x)$ be a two-block parahoric subgroup of $G(\OO_x)$ with blocks of sizes $n_1+n_2=n$, and $P$ the corresponding parabolic. 
Let $W \cong S_n$ be the Weyl group for $\GL_n$, and for a given parabolic subgroup $Q \subset G$ let $W_Q \subset W$ be the Weyl group of its Levi factor. Set $T_0 \coloneqq T(\OO_x)$ and $T_1 \coloneqq \ker(T_0 \to T(\OO_x/\varpi)))$.

Fix $\ov{\rho} : G_x \to \GL_n(k)$ -- a continuous unramified semisimple representation. We say that an irreducible admissible representation $\pi$ of $G$ over $\ov{\F}_p$ is associated to $\ov{\rho}$ if $\pi$ is a subquotient of 
$\ind_B^G \chi_1 \otimes \ldots \otimes \chi_n$ where $\chi_i$ are unramified characters such that $\{\chi_1(\varpi), \ldots, \chi_n(\varpi)\}$
is the set of eigenvalues of $\ov{\rho}(\Frob_x)$. We write $I(\chi)$ for $\ind_B^G \chi_1 \otimes \ldots \otimes \chi_n$. The following lemma shows that if we do not fix the ordering of $\chi_i$, then we can always consider $\pi$ to be a subrepresentation of $I(\chi)$.

\begin{prop}\label{prop:embeds}
    Let $\pi$ be an irreducible admissible $\ov{\F}_p[G]$-module associated to $\ov{\rho}$. Then there exists an ordering of $\chi_1, \ldots, \chi_n$ such that $\pi$ is a subrepresentation of $I(\chi)$.
\end{prop}
\begin{proof}
    We use the adjunction between $\ind_B^G$ and the parabolic restriction $r_B^G$ to get isomorphism 
    \[\Hom(\pi, I(\chi)) \cong \Hom(r_G^B(\pi), \chi).\] Since $\pi$ is associated to $\ov{\rho}$ we know that $r_G^B(\pi) \neq 0$. Since $r_G^B(\pi)$ is a representation of the torus, there exists a $1$-dimensional quotient given by some character $\chi : T \to \ov{\F}_p^\times$. Then we get that $\Hom(\pi, I(\chi)) \neq 0$, and since $\pi$ is irreducible this implies that $\pi$ is a subrepresentation of $I(\chi)$. Then $\chi$ forms the supercuspidal support of $\pi$, and in fact has to be a permutation of the original $\chi_1, \ldots, \chi_n$. For the notion of supercuspidal support in positive characteristic, see \cite[II.2.6]{VigBook}. We would also like to remark here that in the case $q \equiv 1 \pmod p$, $p > n$ the notions of cuspidal and supercuspidal representations coincide, see \cite[II.3.9]{VigBook}.
\end{proof}

We now describe the Bernstein presentation of Iwahori-Hecke algebra $\HH_k(G, \Iw)$, 
following \cite[I.3.14]{VigBook}. Let \[t_j = \diag(\underbrace{\varpi, \ldots \varpi}_{j}, 1 \ldots, 1).\] and set 
$T_j =  [\Iw t_j \Iw]$ and $X^j = T_j(T_{j-1})^{-1}$. We also let $s_j$ be the permutation matrix corresponding to the transposition
$(j, j+1)$, and set $S^j = [\Iw s_j \Iw]$. The elements $X^j$ for $1 \leq j \leq n$ generate the group algebra $k[\Z^n]$ one which 
$S_j$ act by permuting the indices. The Bernstein presentation states that \[\HH_k(G, \Iw) \cong k[S_n \ltimes \Z^n]\] under the action
described above.

Now we introduce some useful Hecke operators. For any ring $R$, $1 \leq i \leq n_1$ and $1 \leq j \leq n_2$ let $V^{j,2} \in \HH_R(G, \p(x))$ be the Hecke operator associated to the double coset 
\[ [ \p(x)\diag(\underbrace{1, \ldots, 1}_{n_1},\underbrace{\varpi, \ldots, \varpi}_{j},\underbrace{1, \ldots, 1}_{n_2-j})\p(x) ]\] and let $V^{i,1}$ be associated to 
\[[\p(x)\diag(\underbrace{\varpi, \ldots, \varpi}_{i}, 1 \ldots, 1)\p(x)]\]

The following is part of \cite[Theorem~B.1]{CHT}
\begin{prop}\label{prop:Iw1_same_Iw}
    Let $V$ be an irreducible admissible $k[G]$-module, which is generated by its Iwahori-invariant vectors. 
    Then $V^{\Iw} = V^{\Iw_1}$.
\end{prop}
Under the conditions of \ref{prop:Iw1_same_Iw}, we thus get an isomorphism
\begin{align}\begin{split}\label{eq:mnH1}
    H^1(\Iw, V) & \cong H^1(B(k), V^{K^1(x)}) \cong H^1(T(k), V^{\Iw_1}) \\ & \cong H^1(T(k), V^{\Iw}) 
    \cong \Hom(T(k), V^{\Iw})
\end{split}\end{align} 

Both sides of \ref{eq:mnH1} can be endowed with the action of $\HH_k(G, \Iw)$. On $H^1(\Iw, V)$ we take the derived $\HH_k(G, \Iw)$-action
and on $\Hom(T(k), V^{\Iw})$ we consider the natural action on the target. 
\begin{prop}\label{prop:equivX}
The isomorphism \ref{eq:mnH1} is equivariant with respect to $X^i$ for all $1 \leq i \leq n$.
\end{prop}
\begin{proof}
The action of $X^i$ on $[f] \in H^1(\Iw, V)$
can be described as follows. Write \[\Iw t_i \Iw = \bigsqcup_j g_{i,j} \Iw.\] 
We now give an explicit description for $g_{i,j}$. Fix a set of representatives $S \subset \OO_F$ for $k$.
For each $m \in M_{i \times (n-i)}(S)$ let $g_{i,m}$ be the matrix such that $g_{i,m}(k, k) = \varpi$ for $k \leq i$,
$g_{i,m}(k, k) = 1$ for $k > i$ and $g_{i,m}(k,\ell) = m(k, \ell - i)$ for $k \leq i, \ell > i$. The rest of the entries are set to $0$.
Let us show that this is a full set of representatives. First we show that $g_{i, m}$ represent distinct cosets, that is 
that $g_{i,m}^{-1}g_{i, m'} \notin \Iw$ for $m \neq m'$. Suppose $m(k, \ell) \neq m'(k, \ell)$. Then 
\[(g_{i,m}^{-1}g_{i, m'})(k, \ell+i) = \varpi^{-1}(m'(k, \ell) - m(k, \ell))\] which is not in $\OO_F$. 
Now we just need to verify that the number of cosets is $q^{i(n-i)}$. Indeed, 
\[[\Iw t_i\Iw : \Iw] = [\Iw : \Iw \cap t_i\Iw t_i^{-1}] = q^{i(n-i)}\] since $\Iw \cap t_i\Iw t_i^{-1}$ are just the elements of the Iwahori
whose $(k, \ell)$-coordinates for $k \leq i, \ell > i$ vanish $\mod \varpi$.

Then 
\[(X^i[f])(x) = \sum_j g_{i,\sigma(j)} f(g_{i,\sigma(j)}^{-1}xg_{i,j})\] where $\sigma$ is the unique permutation such that 
\[g_{i,\sigma(j)}^{-1}xg_{i,j} \in \Iw\] for all $j$. Denote by $\ov{\phantom{x}} : \Iw \to T(k)$ the reduction map.
Let $s$ be the inverse of \ref{eq:mnH1}. For $[\tau] \in \Hom(T(k), V^{\Iw})$ we get
\begin{align*}
    (X^i[s(\tau)])(x) & = \sum_j g_{i,\sigma(j)} s(\tau)(\ov{g_{i,\sigma(j)}^{-1}xg_{i,j}}) \\ & = 
    \sum_j g_{i,\sigma(j)} s(\tau)(\ov{x}) = s(X^i[\tau])(x)
\end{align*}
The second equality is due to all the $g_{i,j}$ being in the Borel and having the same diagonal.
\end{proof}

\begin{defn}
    A $G$-modules $V$ over $\ov{\F}_p$ if \textit{locally admissible} if it is smooth and for every $v \in V$ the subrepresentation generated by $v$ is admissible. Let $\CC$ denote the abelian category of locally admissible $G$-modules $V$ over $\ov{\F}_p$ such that every irreducible subquotient of $V$ is associated to $\ov{\rho}$. 
\end{defn}

The following is analogous to \cite[Lemma~9.14]{CG}:
\begin{prop}\label{prop:enoughinj}
    The category $\CC$ has enough injectives, and the inclusion functor from $\CC$ to locally admissible $G$-modules is exact.
\end{prop}    
\begin{proof}
    Inside the category of $G$-modules, the category $\CC$ is fully contained inside the unipotent block (the block containing the trivial representation). By part 4) of
    \cite[Theorem~B.1]{CHT} the unipotent block is equivalent to the category of $\HH_k(G, \Iw^p)$-modules. Via the Bernstein embedding such modules
    can naturally be viewed as $\HH_k(G, G(\OO_x))$-modules, where $\HH_k(G, G(\OO_x))$ can be explicitly described via the 
    Satake isomorphism as $k[X_1, \ldots, X_n]^W$. If $V$ is any locally admissible element of the unipotent block, 
    the associated Hecke module $V^{\Iw^p}$ is locally finite-dimensional over $k$ and thus we can write
    \[V^{\Iw^p} = \bigoplus_{\m} V^{\Iw^p}_\m\] where the sum is taken over all maximal ideals of $\HH_k(G, G(\OO_x))$.
    Let $\D$ denote the category of locally admissible representations in the unipotent block. Then we can write 
    $\D = \bigoplus_{\m} \D_\m$ where $\D_\m$ consists of representations whose associated $\HH_k(G, G(\OO_x))$-module
    is supported only at $\m$. The maximal ideals of $\HH_k(G, G(\OO_x))$ have the form 
    $(t_1 - \alpha_1, \ldots, t_n - \alpha_n)$ where $\alpha_i \in k$ and $t_i = e_i(X_1, \ldots, X_n)$ is the $i$-th elementary symmetric polynomial of
    $X_1, \ldots, X_n$. If we now let $\mathfrak{n}$ be the ideal defined by $\alpha_i = e_i(\chi_1(\varpi), \ldots, \chi_n(\varpi))$,
    then it is clear that $\CC = \D_{\mathfrak{n}}$. The exactness is now clear, and to show that $\CC$ has enough injectives it is enough to check that the category $\text{Mod}^{\text{l.adm.}}_G(\ov{\F}_p)$ of locally admissible $G$-modules has enough injectives. The full category $\text{Mod}_G(\ov{\F}_p)$ certainly has enough injectives, and the functor\\ $\mathcal{L} : \text{Mod}_G(\ov{\F}_p) \to \text{Mod}^{\text{l.adm.}}_G(\ov{\F}_p)$ taking a module to its smooth locally admissible vectors is right adjoint to the natural embedding $\text{Mod}^{\text{l.adm.}}_G(\ov{\F}_p) \to \text{Mod}_G(\ov{\F}_p)$. This proves the claim.
\end{proof}

From now on, fix $\alpha = \chi_i(\varpi)$ for some $1 \leq i \leq n$.
Let \[P(X) = \prod \limits_{i=1}^n (X - \chi_i(\varpi)).\] 
For $1 \leq j \leq n_2$, let $P_j$ be a polynomial which roots with multiplicities are precisely 
\[\sum_{\substack{J \subset S \\ \#J = j}} \prod_{a \in J} \chi_a(\varpi).\] Factor $P_j = Q_jR_j$ where 
\[R_j(X) = \left(X - \binom{n_2}{j}\alpha^j\right)^{k_j}\] and $Q_j, R_j$ are coprime.  
Set \[e_\alpha \coloneqq \lim_{m \to \infty}\left(\prod \limits_{i=1}^{n_2} Q_j(V^{j,2})\right)^{m!}.\]
Here we consider $e_{\alpha}$ as an operator acting $V^{\p(x)}$ for $V \in \CC$. Since objects in $\CC$ are locally admissible, the limit makes sense.

We now define two functors $F, G : \CC \to \ov{\F}_p\mathbf{-Vect}$. On objects we set 
\[F(V) \coloneqq V^{G(\OO_x)}, \qquad G(V) \coloneqq e_\alpha V^{\p(x)}\]
Note that $F, G$ are both left-exact, and $e_\alpha$ is exact. Then we can form derived functors $R^kF, R^kG$ and identify
\[R^kF(V) = H^k(G(\OO_x), V), \qquad R^kG(V) = e_\alpha H^k(\p(x), V)\]
We have a natural transformation $\iota : F \to G$ given by composing the inclusion $V^{G(\OO_x)} \xhookrightarrow{} V^{\p(x)}$
with $e_\alpha$.

We will make use of the following simple algebraic fact.
\begin{lem}\label{lem:expl_inf_iso}
    Let $G$ be a profinite group $H \triangleleft G$ be a normal subgroup. Let $A$ be a $p$-torsion $G$-module for some positive integer 
    $p$ and let $H$ have pro-$q$ order for a prime $q$ satisfying $q \equiv 1 \pmod p$. Then the inflation map \[\inf : H^1(G/H, A^H) \to H^1(G, A)\] is an isomorphism
    whose inverse sends a cocycle $[f] \in H^1(G, A)$ to \[g \mapsto f(g) + (1-g)a_f\] for some $a_f \in A$.
\end{lem}
\begin{proof}
The condition $q \equiv 1 \pmod p$ insures that $H^1(H, A)$ vanishes. Then it is enough to take $(1-g)a_f$ to be the coboundary trivializing 
the restriction of $[f]$ to $H$.
\end{proof}

\begin{prop}\label{prop:H1inj}
    Let $\pi$ be an irreducible admissible $\ov{\F}_p[G]$-module associated to $\ov{\rho}$. Then the map \[f : H^1(G(k), \pi^{K^1(x)}) \to e_\alpha H^1(P(k), \pi^{K^1(x)})\]
    is injective.
\end{prop}
\begin{proof}
    Both cohomology groups in question inject into $H^1(B(k), \pi^{K^1(x)})$ since $p > n$, so let us analyze that group. Since $q \equiv 1 \pmod p$,
    by inflation-restriction we get \[H^1(B(k), \pi^{K^1(x)}) \cong H^1(T(k), \pi^{\Iw_1}).\] 
    As a special case of \ref{eq:mnH1} we have
    \begin{equation}\label{eq:mnH1sp}
        H^1(\Iw, \pi) \cong H^1(B(k), \pi^{K^1(x)}) \cong \Hom(T(k), \pi^{\Iw}) \cong (\pi^{\Iw})^{\oplus n}.
    \end{equation}
    The isomorphism above is equivariant with respect to the natural actions of $\{X^i\}$ on both sides
    arising from the actions of $\HH_k(G, \Iw)$ by \cref{prop:equivX}. The space $\pi^\Iw$ injects into
    $I(\chi)^\Iw$, which has a basis $\{\varphi_w\}$ for $w \in W$, where $\varphi_w$ is supported on $Bw\Iw$ and satisfies 
    $\varphi_w(w) = 1$. It follows from the proof of \cite[Lemma~5.10]{Thorne12} that on each component of $(I(\chi)^\Iw)^{\oplus n}$ the operator $e_\alpha$ acts as a projection onto the space spanned by $\{\varphi_{w'} \mid w' \in W'\}$, where $W'$ is the subset of $W$ consisting of permutations which send $\{n_1+1, \ldots, n\}$ to the positions of  $\alpha$-s in the sequence $\chi_1(\omega), \ldots, \chi_n(\omega)$. On the level of cocycles, the isomorphism \ref{eq:mnH1sp}
    sends $[s] \in H^1(B(k), \pi^{K^1(x)})$ to the map \[g \mapsto s(g) + (1-g)\psi\]
    for some $\psi \in I(\chi)$ (\cref{lem:expl_inf_iso}).
    Thus a cocycle 
    $[s] \in H^1(G(k), I(\chi)^{K^1(x)})$ being in the kernel of $f$ means that for all $t \in T(k)$ and $w_0 \in W'$ we have
    \begin{equation}\label{eq:co_init}(s(t) + (1-t)\psi)(w_0) = 0.\end{equation} For any $w \in W$ we have
    \[(t\psi)(w) = \psi(w\tilde{t}) = \psi(w(\tilde{t})w) = \psi(w)\]
    Here $\tilde{t}$ is a lift of $t$ to $T_0$, and $w$ acts on the torus in a natural way.
    Thus \begin{equation}\label{eq:cob}((1-t)\psi)(w) = 0.\end{equation} Combining \ref{eq:co_init} and \ref{eq:cob} applied to $w_0$ we get \[s(t)(w_0) = 0.\]
    Now let us conjugate $t$ by an arbitrary $w \in W$. Since the result is again in $T$, we use the cocycle condition
    and the transformation law of $I(\chi)$ with respect to the Borel to write
    \begin{equation}\label{eq:coc1}0 = s(wtw^{-1})(w_0) = (s(w) + w(s(t) + ts(w^{-1})))(w_0)\end{equation}
    \begin{equation}\label{eq:coc2}(wts(w^{-1}))(w_0) = ws(w^{-1})(w_0) = -s(w)(w_0)\end{equation}
    Combining \ref{eq:coc1} and \ref{eq:coc2} we get \[0 = (ws(t))(w_0) = s(t)(w_0w)\]
    In other words, we now have $s(t)(w) = 0$ for all $t \in T(k)$ \textit{and} for all $w \in W$. By \ref{eq:cob} this implies
    that $[s] = 0$ since $\{\varphi_w\}$ make a basis for $I(\chi)^{\Iw}$.
\end{proof}

\begin{thm}\label{thm:c_iso}
The natural transformation $\iota : F \to G$ given by $V^{G(\OO_x)} \mapsto e_\alpha V^{\p(x)}$ on objects is an isomorphism of functors. In particular, we get functorial isomorphisms
\[\iota_* : H^k(G(\OO_x), V) \xrightarrow{\sim} e_\alpha H^k(\p(x), V)\] for all $k \geq 0$.
\end{thm}    
\begin{proof}
    In the proof of \cref{prop:enoughinj} we have identified $\CC$ with a subcategory of $\HH_k(G, \Iw^p)\text{-Mod}$. Thus every element of $\CC$ is a direct limit of finite length elements of $\CC$, and it is therefore enough to establish the isomorphism for finite length $V$. The first step will be to show that 
    $\iota(V)$ is an isomorphism for all $V \in \CC$. For an irreducible subrepresentation $\pi \subset V$
    consider the diagram 
    \begin{equation}\label{eq:five_diag}
        \begin{tikzcd}
            0 \arrow{r}{} & F(\pi) 
            \arrow{r}{} \arrow{d}{\iota(\pi)} & F(V) \arrow{r}{} \arrow{d}{\iota(V)} & F(V/\pi) \arrow{d}{\iota(V/\pi)} \arrow{r}{} & R^1F(\pi) \arrow{d}{f} \\
            0 \arrow{r}{} & G(\pi) 
            \arrow{r}{} & G(V) \arrow{r}{} &
            G(V/\pi) \arrow{r}{} & R^1G(\pi)
        \end{tikzcd}
    \end{equation}
    To show that $\iota(V)$ is injective, we can use the four lemma and induct on the length of $V$. Thus we only need to show that
    $\iota(\pi)$ is injective for irreducible $\pi$. This is done in \cite[Lemma~5.10]{Thorne12}.\\

    Now we would like to show that $\iota(\pi)$ is an isomorphism. Consider the injection $\pi \subset I(\chi)$ 
    and the associated diagram
    \begin{equation}
        \begin{tikzcd}
            0 \arrow{r}{} & F(\pi) 
            \arrow{r}{} \arrow{d}{\iota(\pi)} & F(I(\chi)) \arrow{r}{} \arrow{d}{\iota(I(\chi))} & 
            F(I(\chi)/\pi) \arrow{d}{\iota(I(\chi)/\pi)} \\
            0 \arrow{r}{} & G(\pi) 
            \arrow{r}{} & G(I(\chi)) \arrow{r}{} &
            G(I(\chi)/\pi)
        \end{tikzcd}
    \end{equation}
    We already know that $\iota(I(\chi)/\pi)$ is injective. Then to show that $\iota(\pi)$ is surjective by the four lemma, 
    we need to know that $\iota(I(\chi))$ is surjective. This follows once again from the proof of \cite[Lemma~5.10]{Thorne12}.

    Finally, we are ready to see that $\iota(V)$ is an isomorphism for all $V \in \CC$. We induct on the length of $V$
    using \cref{eq:five_diag}. Since $f$ is injective by \cref{prop:H1inj}, the result follows. 
\end{proof}  
\section{Representation theory of $\GL_n(F_v)$ in characteristic $0$}\label{sec:local}

Fix a finite extension $E/\Q_p$ in $\ov{\Q}_p$ which contains
the images of all embeddings $F \to \ov{\Q}_p$. We write $\OO$ for the ring of integers
of $E$, and  $\varpi \in \OO$ for a choice of uniformizer. If $v$ is a finite place of $F$ prime to $p$, we write $\Xi_v \coloneqq \Z^n$ and $\Xi_{v,1} \coloneqq \langle \tau_v \rangle \times \Z^n$, where $\tau_v$ is the generator of $k_v^\times(p)$ -- the maximal $p$-power order quotient of $k_v^{\times}$. We have a natural homomorphism $\OO_{F_v}^\times \to \Z[\Xi_{v,1}]$ induced by the homomorphism $\OO_{F_v}^\times \to k_v^\times \to k_v^\times(p)$, which we denote by $\langle \, \cdot \, \rangle$.
Consider a standard parabolic subgroup $P \subset \GL_n(F_v)$
corresponding to a partition $n = n_1 + \ldots + n_m$ which we will denote as $\mu$. Given a partition of $n$,
we will always let $s_{\mu, i} = n_1+\ldots+n_i$, with $s_{\mu, 0} = 0$.  Let $P = MN$ and $\ov{P} = M\ov{N}$ be the Levi decompositions of $P$ and its opposite parabolic. 
Let $\m$ be the hyperspecial maximal compact subgroup of $M$. Define the subgroup of the symmetric group $S_\mu = S_{n_1} \times \ldots \times S_{n_m}$. For any positive integer $k$ let \[\Sat_k : \HH_{\Z[q_v^{1/2}]}(\GL_k(
F_v), \GL_k(\OO_{F_v})) \to \Z[q_v^{1/2}][X_1^{\pm 1}, \ldots, X_k^{\pm 1}]^{S_k}\] denote the Satake isomorphism. We use those isomorphisms to identify
\[\Sat_{\mu} = \Sat_{n_1} \otimes \ldots \otimes \Sat_{n_k} : \HH_{\Z[q_v^{1/2}]}(M, \m) \isom \Z[q_v^{1/2}][\Xi_v]^{S_\mu}\]

Consider any open compact subgroup $\mathfrak{q}$ of $\GL_n(F_v)$ and set 
\[\mathfrak{q}_M = \mathfrak{q} \cap M, \quad \mathfrak{q}^+ = \mathfrak{q} \cap N, \quad \mathfrak{q}^- = \mathfrak{q} \cap \ov{N}\]
From now on assume that $\mathfrak{q}$ has an Iwahori decomposition with respect to $P$, which means that $\mathfrak{q} = \mathfrak{q}^- \mathfrak{q}_M \mathfrak{q}^+$.
We define a submonoid $M^+ \subset M$ of \textit{positive} elements to consist of elements $m \in M$ such that 
\[m\mathfrak{q}^+m^{-1} \subset \mathfrak{q}^+, \qquad m^{-1}\mathfrak{q}^-m \subset \mathfrak{q}^-\]
Inside $M^+$ we have a further submonoid $M^{++}$ of \textit{strictly positive} elements consisting of $m \in M^+$ satisfying the following conditions:
\begin{itemize}
\item For any compact open subgroups $\mathfrak{n}_1, \mathfrak{n}_2$ of $N$ there exists a positive integer $x \geq 0$ such that 
    \[m^x\mathfrak{n}_1m^{-x} \subset \mathfrak{n}_2,\]
\item For any compact open subgroups $\ov{\mathfrak{n}}_1, \ov{\mathfrak{n}}_2$ of $N$ there exists a positive integer $x \geq 0$ such that 
\[m^{-x}\ov{\mathfrak{n}}_1m^{x} \subset \ov{\mathfrak{n}}_2.\]
\end{itemize}
We denote by $\HH_{\OO}(M, \mathfrak{q}_M)^+$ the elements of $\HH_{\OO}(M, \mathfrak{q}_M)$ whose support is contained in $M^+$.

\begin{prop}\label{prop:homplus}
\begin{enumerate}
\item The map $t_\mu^+ : \HH_{\OO}(M, \mathfrak{q}_M)^+ \to \HH_{\OO}(G, \mathfrak{q})$ given by 
\[[\mathfrak{q}_Mm\mathfrak{q}_M] \mapsto \delta_P^{1/2}[\mathfrak{q}m\mathfrak{q}]\] is an algebra homomorphism.
\item The map $t_\mu^+$ extends to a homomorphism $t_\mu : \HH_{\OO}(M, \mathfrak{q}_M) \to \HH_{\OO}(G, \mathfrak{q})$ if and only if there exists a totally positive element $a \in Z(M)$ such that $[\mathfrak{q}a\mathfrak{q}]$ is invertible in $\HH_{\OO}(G, \mathfrak{q})$.
\item Assuming the existence of the extension in (2), for any smooth $\C[\GL_n(F_v)]$-module $\pi$, the canonical map $\pi^{\mathfrak{q}} \to \pi_N^{\mathfrak{q}_M}$ is a homomorphism of $\HH_{\OO}(M, \mathfrak{q}_M)$-modules where $\HH_{\OO}(M, \mathfrak{q}_M)$ acts on $\pi^{\mathfrak{q}}$ via the map $t$.
\end{enumerate}
\end{prop}
\begin{proof}
For the first two claims, see \cite[II.6]{VigInd}. For the third, see \cite[II.10.1]{VigInd}
\end{proof}

Now we record some results about smooth admissible representations of $\GL_n(F_v)$ in characteristic 0.
Let $\p$ be a parahoric corresponding to the partition $n = n_1+\ldots+n_k$ which we call $\mu$, and let $P$ be the 
underlying parabolic with the Levi decomposition $P = MN$. Let $\m = M(\OO_{F_v})$.
We also let $\p_1, \m_1$ denote the kernels of the homomorphisms
\[\p \to P(k_v) \to \GL_{n_k}(F_v) \xrightarrow{\det} k_v^{\times} \to k_v^{\times}(p)\]
\[\m \to M(k_v) \to \GL_{n_k}(F_v) \xrightarrow{\det} k_v^{\times} \to k_v^{\times}(p).\]
 Finally, let $\Iw' = \p_1 \cap \Iw$. 

 \begin{lem}
    The condition in part (2) of \cref{prop:homplus} is satisfied for $\mathfrak{q} = \p, \p_1$.  
    \end{lem}
    \begin{proof}
    We will first treat the case of $\p$. Let $e_P = [\p : \Iw]^{-1}[\p]$ be the idempotent in $\HH_{\OO}(G, \Iw)$ corresponding to $\p$. If we identify $\HH_{\OO}(M, \m)$ with $\OO[\Xi_v]^{S_\mu}$, we get a homomorphism 

    \begin{equation}\HH_{\OO}(M, \m) \xhookrightarrow{} \OO[\Xi_v] \to \HH_{\OO}(G, \Iw)\end{equation}
    with the latter morphism coming from the Bernstein presentation for the Iwahori-Hecke algebra (cf. \cite{HKP}).

    For $\lambda \in \Xi_v$ let $\theta_\lambda$ be the image of $\lambda$ in $\HH_{\OO}(G, \Iw)$ under the Bernstein embedding. 
    Let us show that $[\p]$ commutes with $\theta_\lambda$ for any $\lambda \in \Xi_v^{S_\mu}$. Using \cite[Prop.~3.1]{LanskyDec} we can write \[[\p] = \sum_{w \in S_\mu} T_w\] where $T_w = [\Iw w\Iw]$. Thus it is enough to check that $\theta_\lambda$ commutes with $T_w$. The Iwahori-Matsumoto presentation lets us write each $T_w$ as a $\prod T_{s_\alpha}$, where $s_{\alpha}$ are the simple reflections generating $S_\mu \subset W$. Hence we may restrict to $w = s_\alpha$.
    
    The relations of the Bernstein presentation tell us that 
    \[T_{s_\alpha}\theta_\lambda = \theta_\lambda T_{s_\alpha} + (q_v-1)\frac{\theta_{s_\alpha(\lambda)} - \theta_\lambda}{1 - \theta_{-\alpha^\vee}} \]
    and since $s_\alpha(\lambda) = \lambda$ we are done. We therefore get an algebra homomorphism 
    \[\OO[\Xi_v^{S_\mu}] \xrightarrow{- \cdot e_P} \HH_{\OO}(G, \p).\]
    Here we have identified $\HH_{\OO}(G, \p)$ with $e_P\HH_{\OO}(G, \Iw)e_P$. Define a dominant cocharacter $\lambda = (\lambda_1,\ldots, \lambda_n)$ given by $\lambda_i = m - j$ for\\ $s_{\mu, j-1} < i \leq s_{\mu, j}$. If we take $l$ to be the length function on the extended affine Weyl group, then $\OO[\Xi_v^{S_\mu}]$ contains an invertible element $q_v^{l(\lambda)/2} X^\lambda$, where $X^\lambda$ is the element associated to $\lambda$ in $\OO[\Xi_v^{S_\mu}]$. The image of $q_v^{l(\lambda)/2} X^\lambda$ in $\HH_{\OO}(G, \p)$ is the double coset $[\p\diag(\varpi^{\lambda_1}, \ldots, \varpi^{\lambda_n})\p]$, and we are done with the case of $\p$. 
    
    In the case of $\p_1$, we will show note that $[\p : \p_1]^{-1}[\p]$ is central in $\HH_\OO(G, \p_1)$. We can write \[[\p] = \sum_{[M_i] \in \p/\p_1} [\p_1 M_i \p_1],\] so it is enough to choose representatives $M_i$ for $\p/\p_1 \cong k_v^\times(p)$ such that each $[\p_1 M_i \p_1]$ is central in $\HH_\OO(G, \p_1)$. Since $k_v^\times(p)$ has power-$p$ order every element in it is a $n_k$-th power, which means we can choose $M_i$ to be diagonal. The map $\HH_\OO(G, \p) \to \HH_\OO(G, \p_1)$ defined by $[\p : \p_1]^{-1}[\p]$ is therefore an algebra homomorphism, and the result follows from the $\p$ case.
\end{proof}

Fix a uniformizer $\varpi_c$ of $F_v$. For any  $1 \leq j \leq k$ and $1 \leq i \leq n_j$, let \[V^{i,j} = t_\mu(\Sat_{\mu}^{-1}(e_i(X_{s_{\mu, j-1}+1}, \ldots ,X_{s_{\mu,j}}))).\] We will also consider operators in $\HH(G, \p_1)$ such that their actions on $\pi^{\p} \subset \pi^{\p_1}$ agree with the action of $V^{i,j}$ for any smooth representation $\pi$. They can be constructed in the same way as $V^{i,j}$ above by replacing $S_{\mu}$ with the appropriate version of the Satake isomorphism for $\m_1$, to be introduced in the forthcoming work of Dmitri Whitmore. These operators will also be denoted $V^{i,j}$. We also define operators $T^{i,j}$ representing the images of the same elements under $\Sat_{\mu}^{-1}$ in $\HH(M, \m)$ and the corresponding operators on $\HH(M, \m_1)$.

The following lemmas are straightforward generalizations of the lemmas in \cite[Section 5]{Thorne12}. Given a parabolic subgroup $Q$
of $\GL_n(F_v)$, we write $W_Q \subset W$ for the Weyl group of its Levi factor. Recall from \cite{Cas} that the space $W_Q \backslash W / W_P$ has a canonical set of representatives $[W_Q \backslash W / W_P]$, 
consisting of minimal length elements from each double coset.
\begin{lem}\label{lem:partitions}
    Let $Q$ be a parabolic corresponding to the partition $n = m_1 + \ldots + m_r$. Then $[W_Q \backslash W / W_P]$
    is isomorphic to the set of partitions \[m_i = n_1^i + \ldots + n_k^i, 1 \leq i \leq r\]
    such that \[\sum_i n_j^i = n_j \text{ for all } 1 \leq j \leq k.\]
\end{lem}

With $Q$ as in the last lemma, let $L_i$ denote the $i$-th component of the corresponding Levi subgroup. 
For $w \in [W_Q \backslash W / W_P]$ corresponding to the partition $n_1^i + \ldots + n_k^i$, 
let $\p_i^w$ denote the parahoric subgroup of $L_i$ corresponding to this partition, and let 
$\p_{i,1}^w$ be the kernel of 
\[\p_i^w \to \GL_{n_k^i}(F_v)  \xrightarrow{\det} k_v^{\times} \to k_v^{\times}(p)\]
Let $\mathfrak{q}$ be the parahoric corresponding to the partition $\{n_1^1, \ldots, n_k^1, n_1^2, \ldots, n_k^r\}$, 
and let $\n$ be the hyperspecial maximal compact of the corresponding Levi subgroup. 
We define $\p_1^w$ as a subgroup of $\mathfrak{q}$ define by the conditions $\im(\det N_k^j \to k_v^{\times}(p)) = 1$
for all $j$ where $N_k^j$ is the block corresponding to $n_k^j$.

\begin{lem}\label{lem:par_breakdown}
    For each $1 \leq i \leq r$ let $\pi_i$ be a smooth representation of $L_i$. Then 
    \begin{enumerate}
        \item For any $w \in  [W_Q \backslash W / W_P]$ we have $L_i \cap w\p w^{-1} = \p_i^{w}$. 
        \item For any $w \in  [W_Q \backslash W / W_P]$ we have $Q \cap w\p_1 w^{-1} \supset \p_1^w$. 
        \item \[(\ind_Q^G \pi_1 \otimes \ldots \otimes \pi_r)^{\p} \cong \bigoplus_{w \in [W_Q \backslash W / W_P]} \pi_1^{\p_1^w} \otimes \ldots \otimes \pi_r^{\p_r^w}.\]
        \item \[(\ind_Q^G \pi_1 \otimes \ldots \otimes \pi_r)^{\p_1} \subset \bigoplus_{w \in [W_Q \backslash W / W_P]} \pi_1^{\p_{1,1}^w} \otimes \ldots \otimes \pi_r^{\p_{r,1}^w}.\]
    \end{enumerate}
\end{lem}
Let $\pi$ be an irreducible admissible admissible representation of $G$ such that $\pi^{\p_1} \neq 0$. Since $\Iw' \subset \p_1$, supercuspidal
support of $\pi$ consists of tamely ramified characters. We will now use the Bernstein-Zelevinsky classification \cite{BZ}, following the conventions of \cite{Rodier} as they are best suited for applications to local Langlands correspondence. We can write $\pi$ as a quotient of 
\[\ind_Q^G \Sp_{k_1}(\chi_1) \otimes \ldots \otimes \Sp_{k_r}(\chi_r)\] where $\Sp_m(\chi)$ for a tamely ramified character $\chi : F_v^\times \to \C^\times$
is the unique irreducible quotient of $\ind_{B}^{\GL_n} \chi \otimes \chi\lvert \, \cdot \, \rvert \otimes \ldots \otimes \chi \lvert \, \cdot \, \rvert^{n-1}$. 
The twisted Steinberg factors $\Sp_{k_i}(\chi_i)$ correspond to Zelevinsky segments $\Delta_i = (\chi, \chi(1), \ldots, \chi(k_i-1))$.

Let $\mathcal{A}$ index the partitions of $sc(\pi)$ into $k$ labeled subsets $S_1, \ldots, S_k$ 
satisfying the following conditions:
\begin{itemize}
    \item $\lvert S_i \rvert = n_i$ for all $i$.
    \item characters from the same Zelevinsky segment always belong to different subsets.
    \item if $\chi \in S_i, \chi' \in S_j$ share a segment and $\chi' = \chi(a)$ for $a > 0$, then $i < j$. 
\end{itemize}
For each partition $\alpha \in \mathcal{A}$, let $r(\alpha)$ be the representation of $T(F)$ given by tensoring the characters of $sc(\pi)$ 
in the following order: characters in $S_i$ precede characters in $S_j$ when $i < j$, and the ordering of characters within each $S_i$ is induced 
by the ordering of Zelevinsky segments. 

\begin{lem}\label{lem:N_coinv}
    For each $1 \leq i \leq r$ let $\pi_i$ be a smooth representation of $L_i$. Then 
    \[(\ind_Q^G \pi_1 \otimes \ldots \otimes \pi_r)_N^{ss} = \bigoplus_{w \in [W_Q \backslash W / W_P]} \ind_{w^{-1}Qw \cap M}^M w^{-1}(\pi_1 \otimes \ldots \otimes \pi_r)_{L \cap wNw^{-1}}\]
\end{lem}

\begin{lem}\label{lem:steinberg_eigenvalues}
    Let $\pi$ be an irreducible admissible $GL_{n}(F_v)$-module such that $\pi^{\p_1} \neq 0$. Consider $\pi^{\p_1}$ as a module $\Z[\Xi_v]^{S_\mu}$-module via the map $t_\mu \circ \Sat_{\mu}^{-1}$. Then $(\pi^{\p_1})^{ss}$ is a direct sum of 1-dimensional submodules indexed by a subset of $\mathcal{A}$. For a finite set 
    $S$ of characters and positive integer $k \leq \lvert S\rvert$, let $e_k(S(\varpi))$ denote the $k$-th symmetric polynomial of elements of $S$ evaluated at $\varpi$. Then on the component associated to $(S_1, \ldots, S_k) \in \mathcal{A}$ the action of $V^{i,j}$ is given by $e_i(S_j)$ for all $1 \leq i \leq n_j$.
\end{lem}
\begin{proof}
    We have a surjection \[\ind_Q^G \Sp_{k_1}(\chi_1) \otimes \ldots \otimes \Sp_{k_r}(\chi_r) \twoheadrightarrow \pi,\]
    and the induced map 
    \[(\ind_Q^G \Sp_{k_1}(\chi_1) \otimes \ldots \otimes \Sp_{k_r}(\chi_r))^{\p_1} \rightarrow \pi^{\p_1}\]
    is also surjective.
    By \cref{lem:N_coinv} we can write 
    \begin{align*}& (\ind_Q^G \Sp_{k_1}(\chi_1) \otimes \ldots \otimes \Sp_{k_r}(\chi_r))_N^{ss} = \\ & \sigma \oplus \bigoplus_{(S_1, \ldots, S_k) \in \mathcal{A}}
    \ind_{B \cap M}^M \left(\bigotimes_{\psi_1 \in S_1} \psi_1 \otimes \ldots \otimes \bigotimes_{\psi_k\in S_k} \psi_k \right).\end{align*}
    Here the summands indexed by $\mathcal{A}$ correspond to $w \in [W_Q \backslash W / W_P]$ represented by partitions $\{n^i_j\}$ satisfying $n^i_j \leq 1$ for all $i,j$ (cf. \cref{lem:partitions}), and $\sigma$ represents all other summands. We will now show that $\sigma$ does not have $\m_1$-invariants. 
    Let $\m_{i,1}^w \subset \p_{i,1}^w$ be the subgroups of the Levi subgroup of $L_i$ defined analogously to $\p_{i,1}^w$.

    Suppose $\sigma^{\m_1}$ is nonzero. Let $\theta$ be a representation of $GL_{n^i_j}(F_v)$ which is a tensor factor of $(\Sp_{k_1}(\chi_1) \otimes \ldots \otimes \Sp_{k_r}(\chi_r))_{L \cap wNw^{-1}}$ for some $w \in [W_Q \backslash W / W_P]$ contributing to $\sigma$. Then $\theta$ has to be spherical if $j < k$, and has to have a fixed vector by $\ker(GL_{n^i_j}(\OO_{F_v}) \to GL_{n^i_j}(k_v) \xrightarrow{\det} k_v^\times \to k_v^\times(p))$ if $j = k$. This would imply that $\Sp_{k_i}(\chi_i)^{\p_{i,1}^w} \neq 0$ for all $1 \leq i \leq r$ and all $w$ representing partitions $m_i = n_1^i +\ldots+ n_k^i$ such that there exists at least one $1 \leq i \leq r$ for which $k_i > 1$ and $n_j^i > 1$ for some $1 \leq j \leq k$. To get a contradiction it is therefore enough to show that $\Sp_{k_i}(\chi_i)^{\p_{i,1}^w} = 0$.

    Define the subgroup $\Iw_i^' \subset {\p_{i,1}^w}$ to be a subgroup of the $L_i$-Iwahori with $1$'s $\mod \varpi$
    on the diagonal at indices $n_{k-1}^i+1$ through $n_k^i$. There are two possibilities: either $\p_{i,1}^w = \GL_{m_i}(\OO_{F_v})$, 
    or $\Iw_i^'$ has at least one $* \mod \varpi$ on the diagonal. In the former case we are done since $\Sp_{k_i}(\chi_i)$ is never spherical.
    In the latter case, let $\mathfrak{t}'$ be the diagonal component of $\Iw_i^'$. Then 
    \[\Sp_{k_i}(\chi_i)^{\Iw_i^'} = \Sp_{k_i}(\chi_i)_U^{\mathfrak{t}'} = (\chi_i \otimes \ldots \otimes \chi_i\lvert \, \cdot \, \rvert^{k_i-1})^{\mathfrak{t}'}\]
    where $U$ is the unipotent radical of the Borel. Since $\mathfrak{t}'$ has at least one $\OO_{F_v}^{\times}$ factor, if this is nonzero $\chi_i$ must be unramified. But in this case any $\p_{i,1}^w$-fixed vector
    would be automatically fixed by the parahoric $\p_i^w$, which properly contains the Iwahori and hence does not fix any vector in $\Sp_{k_i}(\chi_i)$.
\end{proof}

For a partition $n = n_1+\ldots+n_k$ which we call $\mu$, define elements
\[P_{\mu, i} = \prod \limits_{j = s_{\mu,i-1}+1}^{s_{\mu,i}} (T - X_j)\] 
\[\res_{\mu} = \prod_{i < j} \res( P_{\mu,i}, P_{\mu,j}) \in \Z[\Xi_v]^{S_\mu}\]
\[\res_{q_v, \mu} = \prod_{i < j} \res(P_{\mu,i}(q_vT), P_{\mu,j}) \in \Z[\Xi_v]^{S_\mu}\]
Then there exist unique polynomials $Q_{\mu, i} \in \Z[\Xi_v]^{S_\mu}[T]$ such that $\deg Q_i < \mu_i$ and 
\[\sum \limits_{i=1}^n Q_{\mu,i}\prod_{j \neq i} P_{\mu, j} = \res_\mu.\] Define 
\[E_{\mu, i} = Q_{\mu,i} \prod_{j \neq i} P_{\mu, j}.\]
The following statement is elementary.
\begin{lem}
    Take any $A \in M_n(\C)$ with a factorization \[\det(T - A) = \prod \limits_{i=1}^k p_{\mu, i}(T)\] where $p_{\mu, i} \in \C[T]$ are pairwise coprime. Consider the homomorphism $\varphi : \Z[\Xi_v]^{S_\mu} \to \C$ defined by the polynomials $p_{\mu, i}$. By this we mean the homomorphism sending $e_j(X_{s_{\mu,i-1}+1}, \ldots, X_{s_{\mu,i}})$ to $(-1)^j$ times the coefficient of $p_{\mu, i}$ by $T^j$. This homomorphism can be extended to $\varphi : \Z[\Xi_v]^{S_\mu}[T, \res_\mu^{-1}] \to \C[T]$. Then $\varphi(E_{\mu, i}/\res_\mu)(A)$ projects $\C^n$ onto the sum of generalized eigenspaces of $A$ corresponding to the roots of $p_{\mu, i}$.
\end{lem}

\begin{prop}\label{prop:res_or}
    Let $\pi$ be an irreducible admissible $GL_{n}(F_v)$-module. Then either $\res_{q_v, \mu}^{n!}\pi^{\p_1} = 0$, or \[\rec^T_{F_v}(\pi) = (\chi_1\oplus \ldots \oplus \chi_n, 0)\] where $\chi_1,\ldots,\chi_{n_1+\ldots+n_{k-1}}$ are unramified and the rest are tamely ramified with equal restriction to inertia.
\end{prop}
\begin{proof}
Using the notation from the discussion preceding \cref{lem:N_coinv}, if there exists some $k_i > 1$, then $\res_{q_v, \mu}^{n!}\pi^{\p_1} = 0$ follows from \cref{lem:steinberg_eigenvalues}. Otherwise we can apply the proof of \cite[Lemma~3.1.6]{CHT} for the second conclusion.
\end{proof}

\begin{prop}\label{prop:res_or_vee}
    Let $\pi$ be an irreducible admissible $GL_{n}(F_v)$-module. Let $(r, N) = \rec^T_{F_v}(\pi)$. Then either $(t_{\mu}^{-1} \circ \iota_{\HH} \circ t_\mu)(\res_{q_v, \mu}^{n!})\pi^{\p_1} = 0$, or $N = 0$ and \[r^\vee = \chi_1\oplus \ldots \oplus \chi_n\] where $\chi_1,\ldots,\chi_{n_1+\ldots+n_{k-1}}$ are unramified and the rest are tamely ramified with equal restriction to inertia.
\end{prop}
\begin{proof}
Let $\pi^\vee$ be the contragradient of $\pi$. Then $\rec_{F_v}(\pi^\vee) = (r^\vee, -^tN)$. We have a perfect pairing $(\pi^\vee)^{\p_1} \times \pi^{\p_1} \to \C$ which is anti-symmetric with respect to action of $\OO[\Xi_{v,1}]^{S_\mu}$ and $t_{\mu}^{-1} \circ \iota_{\HH} \circ t_\mu$. Therefore $(t_{\mu}^{-1} \circ \iota_{\HH} \circ t_\mu)(\res_{q_v, \mu}^{n!})\pi^{\p_1} = 0$ if and only if $\res_{q_v, \mu}^{n!}(\pi^\vee)^{\p_1} = 0$. Thus we can assume both of these are nonzero, in which case by \cref{prop:res_or} we get the desired result. 
\end{proof}

Let $\varphi_v \in G_{F_v}$ be any lift of Frobenius.
\begin{prop}\label{prop:relation_1_rep}
    Let $\pi$ be an irreducible admissible $GL_{n}(F_v)$-module. Let $(r, N) = \rec^T_{F_v}(\pi)$. Let $R$ be the image of $\OO[\Xi_{v,1}]^{S_\mu}$ in $\End_\OO(\pi^{\p_1})$ under the map $t_\mu$. Then either $\res_{q_v, \mu}^{n!}\pi^{\p_1} = 0$ or the following relation holds over $R$ : for all $\tau \in I_{F_v}$ \[\res_{\mu}^{n!}\left(\sum \limits_{i = 1}^{k-1} E_{\mu, i}(r(\varphi_v)) + \langle \textnormal{Art}_{F_v}^{-1}(\tau) \rangle E_{\mu, k}(r(\varphi_v)) - \res_{\mu}r(\tau) \right) = 0.\]
\end{prop}
\begin{proof}
Assume $\res_{q_v, \mu}^{n!}\pi^{\p_1} \neq 0$. It is enough to check our relation for each localization of $R$ at a maximal ideal $\m$. If $\res_{\mu} \in \m$, then $\res_{\mu}^{n!} = 0$ in $R_\m$. Otherwise $R_\m = \C$ by \cite[\href{https://stacks.math.columbia.edu/tag/00UA}{Tag 00UA}]{stacks-project}, and the image of $\OO[\Xi_{v,1}]^{S_\mu}$ in $R/\m$ corresponds to the polynomials $\prod \limits_{j = s_{\mu, i-1}+1}^{s_{\mu, i}}(T - \chi_j(\varphi_v))$ for $i = 1, \ldots, k$. Then the image of \[{\res_\mu}^{-1}\left( \sum \limits_{i = 1}^{k-1} E_{\mu, i}(r(\varphi_v)) + \langle \text{Art}_{F_v}^{-1}(\tau) \rangle E_{\mu, k}(r(\varphi_v)) \right)\] in $M_n(R_\m)$ is a diagonal matrix with $n - n_k$ first entries equal to $1$ and the rest equal to $\chi_n(\tau)$. This concludes the proof.
\end{proof}

\begin{prop}\label{prop:relation_1_rep_vee}
    Let $\pi$ be an irreducible admissible $GL_{n}(F_v)$-module. Let $(r, N) = \rec^T_{F_v}(\pi)$. Let $R'$ be the image of $\OO[\Xi_{v,1}]^{S_\mu}$ in $\End_\OO(\pi^{\p_1})$ via the map $t_{\mu}^{-1} \circ \iota_{\HH} \circ t_\mu$. Then either $(t_{\mu}^{-1} \circ \iota_{\HH} \circ t_\mu)(\res_{q_v, \mu}^{n!})\pi^{\p_1} = 0$ or the following relation holds over $R'$ : for all $\tau \in I_{F_v}$ \[(t_{\mu}^{-1} \circ \iota_{\HH} \circ t_\mu)\left(\res_{\mu}^{n!}\left(\sum \limits_{i = 1}^{k-1} E_{\mu, i}(r^\vee(\varphi_v)) + \langle \textnormal{Art}_{F_v}^{-1}(\tau) \rangle E_{\mu, k}(r^\vee(\varphi_v)) - \res_{\mu}r^\vee(\tau) \right)\right) = 0.\]
\end{prop}
\begin{proof}
    This follows from \cref{prop:res_or_vee} in the same way as \cref{prop:relation_1_rep} follows from \cref{prop:res_or}. 
\end{proof} 
\section{Setup}\label{sec:setup}

Let $F/F^+$ be an imaginary CM-field with ring of integers $\OO$. Let $\Psi_n$ be the matrix with $1$-s on the antidiagonal and $0$-s elsewhere, and let 
\[J_n = \begin{pmatrix} 0 & \Psi_n \\ -\Psi_n & 0 \end{pmatrix}\]
Define $\til{G}$ to be the group scheme over $\OO_{F^+}$ defined by the functor of points 
\[ \til{G}(R) = \{g \in \GL_{2n}(R \otimes_{\OO_{F^+}} \OO_{F}) \mid {}^tg J_n g^c = J_n\}\]
Then $\til{G}$ is a quasi-split reductive group over $F^+$. It is a form of $\GL_{2n}$ which becomes 
split after the quadratic base change $F/F^+$. If $v$ is a place of $F$ lying above a place $\ov{v}$ of $F^+$
which splits in $F$, then we have a canonical isomorphism $\iota_v : \til{G}(F^+_{\ov{v}}) \cong \GL_{2n}(F_v)$. There is an isomorphism $F_{\ov{v}}^+ \otimes_{F^+} F \cong F_v \times F_{v^c}$, and $\iota_v$ is given by composition \[\til{G}(F_{\ov{v}^+}) \xhookrightarrow{} \GL_{2n}(F_v) \times \GL_{2n}(F_{v^c}) \to \GL_{2n}(F_v)\] where the second maps is the projection on the first factor.
We write $T \subset B \subset G$ for the subgroups consisting, respectively, of the diagonal and upper-triangular matrices in $\til{G}$. Similarly we write $G \subset P \subset \til{G}$ for the Levi and parabolic subgroups consisting, respectively, of the block upper
diagonal and block upper-triangular matrices with blocks of size $n \times n$. Then
$P = U \rtimes G$, where $U$ is the unipotent radical of $P$, and we can identify $G$ with
$\res_{\OO_F / \OO_{F^+}} \GL_n$ via the map 
\[\begin{pmatrix} A & 0 \\ 0 & D \end{pmatrix} \mapsto D \in \GL_n(R \otimes_{\OO_{F^+}} \OO_F)\]

An element $(g_v)_v \in G(\A_{F^+}^\infty) = 
GL_n(\A_F^\infty)$ is called \textit{neat} if the intersection $\cap_v\Gamma_v$
is trivial, where $\Gamma_v \subset \ov{\Q}^\times$ is is the torsion subgroup of the subgroup of 
$\ov{F_v}^\times$ generated by the eigenvalues of $g_v$ acting via some faithful representation of $G$.
We call a neat open compact subgroup $K \subset G(\A_{F^+}^\infty)$ \textit{good} if it has the from
$K = \prod_v K_v$, where the product is running over the finite places of $F$. We make similar definitions with $\til{G}$ in place of $G$.

After extending scalars to $F^+$, $T$ and $B$ form a maximal torus and a Borel subgroup, respectively, of $\til{G}$, and $G$ is the unique Levi subgroup of the parabolic subgroup $P$ of $\til{G}$ which contains $T$.
We call an open compact subgroup $\til{K}$ of $\til{G}(\A_{F^+}^\infty)$ \textit{decomposed} with respect to the Levi decomposition $P = GU$ if $\til{K} = \til{K}_{G} \ltimes \til{K}_{U}$ where $\til{K}_{G}$ is the image of $\til{K}$ in $G$ and $\til{K}_{U} = \til{K} \cap U(\A_{F^+}^\infty)$.

If $K$ is a good subgroup of $G$ we let $X_K$ be the corresponding locally symmetric space. 
Similarly, if $\til{K}$ is a good open compact subgroup of $\til{G}$, then $\til{X}_{\til{K}}$ denotes 
the locally symmetric space. More generally, if $H$ is a connected algebraic group over a number field $L$ and $K_H \subset H(\A_M^\infty)$ is a good subgroup, then we write $X_{K_H}^H$ for the locally symmetric space of $H$ of level $K_H$.

Fix a rational prime $p$, and a finite extension $E/\Q_p$ which contains the images 
of all embeddings $F \xhookrightarrow{} \ov{\Q}_p$. We write $\OO$ for the ring of integers
of $E$, and  $\varpi \in \OO$ for a choice of uniformizer. For $\lambda \in (\Z_+^n)^\Hom(F, E)$ we 
define an $\OO[\prod_{v \mid p} \GL_n(\OO_{F_v})]$-module $\V_\lambda$ as in \cite[\S 2.2.1]{10author}. Similarly for 
$\til{\lambda} \in (\Z_+^{2n})^\Hom(F^+, E)$ we have an $\OO[\prod_{\ov{v} \mid p} \til{G}(\OO_{F^+_{\ov{v}}})]$-module $\V_{\til{\lambda}}$. Both $\V_\lambda$ and $\V_{\til{\lambda}}$ are finite free $\OO$-modules. 

Let $S$ be a set of places of $F$ such that $S = S^c$ and such that $S$ contains all places above $p$ and all places of $F$ which are ramified over $F^+$. Let $\ov{S}$ be the set of places of $F^+$ lying below a place in $S$. Let $K \subset G(\A_{F^+}^\infty)$ be a good subgroup such that $K_{\ov{v}} = G(\OO_{F^+_{\ov{v}}})$ for $\ov{v} \notin \ov{S}$, and similarly let $\til{K} \subset \til{G}(\A_{F^+}^\infty)$ be a good subgroup such that $\til{K}_v = \til{G}(\OO_{F^+_{\ov{v}}})$ for $\ov{v} \notin \ov{S}$.

Define the Hecke algebras 
\[\HH^S = \HH_{\OO}(G(\A_{F^+}^{\infty, \ov{S}}), K^{\ov{S}})\]
\[\til{\HH}^S = \HH_{\OO}(\til{G}(\A_{F^+}^{\infty,\ov{S}}), \til{K}^{\ov{S}})\]
\[\TT^S \cong \bigotimes'_{v \notin S} \OO[\Xi_v]^{S_n}\]
\[\til{\TT}^S \cong \bigotimes'_{v \notin S} \OO[\Xi_{\ov{v}}]^{S_{2n}}\]
Using the isomorphism
\[G(\OO_{F^+_{\ov{v}}}) \cong \GL_n(\OO_{F_v})\]
together with the Satake isomorphisms, as well as an homomorphism $\OO[\Xi_{\ov{v}}]^{S_{2n}} \to \HH_\OO(\til{G}(F^+_{\ov{v}}), \til{G}(\OO_{F^+_{\ov{v}}}))$ given by the polynomial $\til{P}_v(X)$ defined in \cite[Eq.~2.2.5]{10author},
we get homomorphisms $\TT^S \to \HH^S$ and $\til{\TT}^S \to \til{\HH}^S$.
We also have homomorphisms 
\[\TT^S \to \End_{\mathbf{D}(\OO)}(R\Gamma(X_K, \V_\lambda))\] 
\[\til{\TT}^S \to \End_{\mathbf{D}(\OO)}(R\Gamma(\til{X}_{\til{K}}, \V_{\til{\lambda}}))\]
defined in \cite[\S 2.1.2]{10author} and we can denote by $\TT^S(K, \lambda)$, $\til{\TT}^S(\til{K}, \til{\lambda})$ respectively the images of those homomorphisms. The functor $H^*$ induces surjective $\OO$-algebra homomorphisms 
\[ \TT^S(K, \lambda) \to \End_{\OO}(H^*(X_K, \V_\lambda))\] 
\[\til{\TT}^S(\til{K}, \til{\lambda}) \to  \End_{\OO}(H^*(\til{X}_{\til{K}}, \V_{\til{\lambda}}))\]

\section{Boundary cohomology}

Let $\til{K} \subset \til{G}(\A_{F^+}^\infty)$ be a neat compact open subgroup decomposed with respect to the Levi decomposition $P = GU$. We also assume that $\til{K}_v = \til{G}(\OO_{F^+_{\ov{v}}})$ for $\ov{v} \notin \ov{S}$. Define $K$ as the image of $\til{K}$ in $G(\A_{F^+}^\infty)$, $\til{K}_P = \til{K} \cap P(\A_{F^+}^\infty)$ and $K_U = \til{K} \cap U(\A_{F^+}^\infty)$. Both $K$ and $\til{K}_P$ are neat. We recall from \cite[\S 3.1.2]{NT} that the boundary $\partial \til{X}_{\til{K}} = \ov{\til{X}}_{\til{K}}$ of the Borel-Serre compactification has a $\til{G}(\A_{F^+}^\infty)$-equivariant stratification indexed by the standard parabolic subgroups of $\widetilde{G}$. For each standard parabolic subgroup $Q$, label the corresponding stratum $\til{X}^Q_{\til{K}}$. We can write 
\[\til{X}^Q_{\til{K}} = Q(F^+)\backslash (X^Q \times \til{G}(\A_{F^+}^\infty)/\til{K}).\]

From now on, we will focus on the stratum $\til{X}^P_{\til{K}}$ corresponding to the Siegel parabolic. Let us establish some useful maps between the manifolds introduced above. The stratum $\til{X}^P_{\til{K}}$ can be described as a union of connected components indexed by the set $P(F^+)\backslash \til{G}(\A_{F^+}^\infty)/\til{K}$. The locally symmetric space $X^P_{\til{K}}$ is a union of the same components indexed by the set $P(F^+)\backslash P(\A_{F^+}^\infty)/\til{K}_P$. Thus we have a natural open immersion $i : X^P_{\til{K}} \to \til{X}^P_{\til{K}}$ such that $i^* : H^*(\til{X}^P_{\til{K}}, \OO) \to H^*(X^P_{\til{K}}, \OO)$ is a split epimorphism. We also have a proper map $j : X^P_{\til{K}_P} \to X_K$ which has a section by \cite[\S 3.1.1]{NT}. Thus we get a split monomorphism $j^* : H^*(X_K, \OO) \to H^*(X^P_{\til{K}}, \OO) $. We also recall the "restriction to P" and "integration along N" homomorphisms:
\[r_P : \HH_{\OO}(\til{G}(\A_{F^+}^{\infty,\ov{S}}), \til{K}^{\ov{S}}) \to \HH_{\OO}(\til{P}(\A_{F^+}^{\infty,\ov{S}}), \til{K}_P^{\ov{S}})\]

\[r_G : \HH_{\OO}(\til{P}(\A_{F^+}^{\infty,\ov{S}}), \til{K}_P^{\ov{S}}) \to \HH_{\OO}(G(\A_{F^+}^{\infty, \ov{S}}), K^{\ov{S}})\]
defined in \cite[\S 2.2]{NT}. We record the following proposition:
\begin{prop}\label{prop:og_coh}
\begin{enumerate}
\item For all $t \in \til{\TT}^S$ and $h \in H^*(\til{X}^P_{\til{K}}, \OO)$ we have $i^*(th) = r_P(t)i^*(h)$. 
\item For all $t \in \HH_{\OO}(\til{P}(\A_{F^+}^{\infty,\ov{S}}), \til{K}_P^{\ov{S}})$ and $h \in H^*(X_K, \OO)$ we have $j^*(r_G(t)h) = tj^*(h)$. 
\end{enumerate}
\end{prop}

Consider the composite \[\mathcal{S} = r_G \circ r_P : \HH_{\OO}(\til{G}(\A_{F^+}^{\infty,\ov{S}}), \til{K}^{\ov{S}}) \to \HH_{\OO}(G(\A_{F^+}^{\infty, \ov{S}}), K^{\ov{S}}).\] By \cite[Proposition-Definition~5.3]{NT} this map coincides with the tensor product of maps $\OO[\Xi_{\ov{v}}]^{S_{2n}} \to \OO[\Xi_v]^{S_{n}}$ determined by the polynomial $\Sat_n(P_v(X)q_v^{n(2n-1)}P_{v^c}^\vee(q_v^{1-2n}X))$.

Let $\m \subset \TT^S$ be a non-Eisenstein maximal ideal of Galois type with residue field $k$. We have an associated continuous semisimple representation $\ov{\rho}_\m : G_{F, S} \to \GL_n(k)$ such that $\det(X - \ov{\rho}_\m(\Frob_v)) \equiv P_v(X) \mod \m$. Fix a tuple $(Q, (\alpha_v)_{v \in Q})$ where
\begin{itemize}
    \item $Q \subset S$ and $Q \cap Q^c = \varnothing$.
    \item Each place $v \in Q$ is split over $F^+$. Moreover, for each place $v \in Q$ there exists an imaginary quadratic subfield $F_0 \subset F$ such that $q_v$ splits in $F_0$. 
    \item For each place $v \in Q$, $\ov{\rho}_\m$ is unramified at $v$ and $v^c$ and $\alpha_v$ is a root of $\det(X - \ov{\rho}_\m(\Frob_v)))$.
\end{itemize}
For each $v \in Q$, let $d_v$ be multiplicity of $\alpha_v$ as a root of $\det(X - \ov{\rho}_\m(\Frob_v))$. Fix the partitions 
\[\mu_v : 2n = d_v + (n-d_v) + n\]
\[\nu_v : n = d_v + (n-d_v)\]
Let \[\Delta_v = \bigsqcup_{m \in M_{\mu_v}^+} [\p_{\mu_v,1}m\p_{\mu_v,1}] \subset \GL_n(F_v)\]

Now we recall the theory of Hecke algebras of a monoid from \cite[\S 2.1.9]{10author}. Specifically, we consider the restriction from $\til{G}$ to $P$
\[r_P : \HH(\iota_v^{-1}(\Delta_v), \iota_v^{-1}(\p_{\mu_v,1})) \to \HH(P(F_{\ov{v}}^+), P(F_{\ov{v}}^+) \cap \iota_v^{-1}(\p_{\mu_v,1}))\]
and integration along fibers 
\[r_G : \HH(P(F_{\ov{v}}^+), P(F_{\ov{v}}^+) \cap \iota_v^{-1}(\p_{\mu_v,1}) \to \HH(G(F_{\ov{v}}^+), G(F_{\ov{v}}^+) \cap \iota_v^{-1}(\p_{\mu_v,1}))\]
combined with the isomorphism
\[\HH(G(F_{\ov{v}}^+), G(F_{\ov{v}}^+) \cap \iota_v^{-1}(\p_{\mu_v,1})) \cong \HH(\GL_n(F_v) \times \GL_n(F_{v^c}), \p_{\nu_v,1} \times \GL_n(\OO_{F_{v^c}}))\]
we get a Satake transform 
\[\mathcal{S}_v^+ : \HH(\iota_v^{-1}(\Delta_v), \iota_v^{-1}(\p_{\mu_v,1})) \to \HH(\GL_n(F_v) \times \GL_n(F_{v^c}), \p_{\nu_v,1} \times \GL_n(\OO_{F_{v^c}}))\]
By the same argument as in the proof of \cref{prop:homplus} we see that $\mathcal{S}_v^+$ extends to a homomorphism
\[\mathcal{S}_v : \HH(\til{G}(F_{\ov{v}}^+), \iota_v^{-1}(\p_{\mu_v,1})) \to \HH(\GL_n(F_v) \times \GL_n(F_{v^c}), \p_{\nu_v,1} \times \GL_n(\OO_{F_{v^c}}))\]
This homomorphism fits into a commutative diagram 
\[
\begin{tikzcd}
    \OO[\til{\Xi}_{\ov{v},1}]^{S_{\mu_v}} \arrow{r} \arrow{d}{\mathcal{S}_v^f} & \HH(\til{G}(F_{\ov{v}}^+), \iota_v^{-1}(\p_{\mu_v,1})) \arrow{d}{\mathcal{S}_v} \\
    \OO[\Xi_{v,1}]^{S_{\nu_v}} \otimes_{\OO} \OO[\Xi_{v^c}]^{S_n} \arrow{r} & \HH(\GL_n(F_v) \times \GL_n(F_{v^c}), \p_{\nu_v,1} \times \GL_n(\OO_{F_{v^c}}))
\end{tikzcd}    
\]
where $\mathcal{S}_v^f$ is the unique homomorphism which corresponds the polynomial $\prod_{i=1}^{2n}(T - X_i)$ to the tuple of polynomials $\prod_{i=1}^{d_v}(T - X_i), \prod_{i=d_v+1}^{n}(T - X_i), \Sat_n(q_v^{n(2n-1)}P_{v^c}^\vee(q_v^{1-2n}X))$ and maps $\tau_{\ov{v}}$ to $\tau_v$.

We can define global Hecke algebras associated to our Taylor-Wiles data:
\[\til{\HH}^S_Q = \til{\HH}^S \otimes_{\Z} \bigotimes_{v \in Q} \HH(\til{G}(F_{\ov{v}}^+), \iota_v^{-1}(\p_{\mu_v,1}))\]
\[\til{\TT}^S_Q = \til{\TT}^S \otimes_{\Z} \bigotimes_{v \in Q} \OO[\til{\Xi}_{\ov{v},1}]^{S_{\mu_v}}\]
\[\HH^S_Q = \HH^S \otimes_{\Z} \bigotimes_{v \in Q} \HH(\GL_n(F_v) \times \GL_n(F_{v^c}), \p_{\nu_v,1} \times \GL_n(\OO_{F_{v^c}}))\]
\[\TT^S_Q = \TT^S \otimes_{\Z} \bigotimes_{v \in Q} \OO[\Xi_{v,1}]^{S_{\nu_v}} \otimes_{\OO} \OO[\Xi_{v^c}]^{S_n}\]

We can conclude this section with the following proposition:
\begin{prop}
    There exist homomorphisms $\mathcal{S}_Q^f : \til{\TT}^S_Q \to \TT^S_Q$ and $\mathcal{S}_Q : \til{\HH}^S_Q \to \HH^S_Q$ fitting into a commutative diagram 
    \[
    \begin{tikzcd}
        \til{\TT}^S_Q \arrow{r} \arrow{d}{\mathcal{S}_Q^f} & \til{\HH}^S_Q \arrow{d}{\mathcal{S}_Q} \\
        \TT^S_Q \arrow{r} & \HH^S_Q
    \end{tikzcd}    
    \]
    where $\mathcal{S}_Q^f$ coincides with $\mathcal{S}_v^f$ at places $v \in Q$ and with the Satake isomorphism from \cite[Proposition-Definition~5.3]{NT} at places $v \notin S$.
\end{prop}

Let $\til{K}$ be a good subgroup of $\til{G}(\A_{F^+}^\infty)$ such that $\til{K}^S = \til{G}(\widehat{\OO}_{F^+}^{\ov{S}})$ and $\til{K}$ is decomposed with respect to $P$. We can define subgroups $\til{K}_1(Q) \subset \til{K}_0(Q) \subset \til{K}$ as follows:
\begin{itemize}
    \item If $\ov{v} \notin \ov{Q}$, then $\til{K}_1(Q)_{\ov{v}} = \til{K}_0(Q)_{\ov{v}} = \til{K}_{\ov{v}}$.
    \item If $\ov{v} \in \ov{Q}$, then $\til{K}_1(Q)_{\ov{v}} = \iota_v^{-1}(\p_{\mu_v,1})$ and $\til{K}_0(Q)_{\ov{v}} = \iota_v^{-1}(\p_{\mu_v})$.
\end{itemize}

Let $K_1(Q), K_0(Q), K$ be the images in $G(\A_{F^+}^\infty)$ of the intersections of $\til{K}_1(Q), \til{K}_0(Q), \til{K}$ with $P(\A_{F^+}^\infty)$. From the definition we can see that all the subgroups from the previous sentence are decomposed with respect to $P$.
\begin{prop}
For $i = 0,1$ we have
\begin{enumerate}
\item The open immersion $i : X^P_{\til{K}_i(Q)} \to \til{X}^P_{\til{K}_i(Q)}$ yields a split epimorphism\\ $i^* : H^*(\til{X}^P_{\til{K}_i(Q)}, \OO) \to H^*(X^P_{\til{K}_i(Q)}, \OO)$.
\item The proper map $j : X^P_{\til{K}_i(Q)_P} \to X_{K_i(Q)}$ yields a split monomorphism\\ $j^* : H^*(X_{K_i(Q)}, \OO) \to H^*(X^P_{\til{K}_i(Q)}, \OO)$
\item For all $t \in \HH_{\OO}(\iota_v^{-1}(\Delta_v), \iota_v^{-1}(\p_{\mu_v,1}))$ and $h \in H^*(\til{X}^P_{\til{K}_i(Q)}, \OO)$ we have\\ $i^*(th) = r_P(t)i^*(h)$. 
\item For all $t \in \HH_{\OO}(\til{P}(\A_{F^+}^{\infty,\ov{S}}), {\til{K}_i(Q)}_P^{\ov{S}})$ and $h \in H^*(X_{K_i(Q)}, \OO)$ we have\\ $j^*(r_G(t)h) = tj^*(h)$. 
\end{enumerate}
\end{prop}
\begin{proof}
    This follows from the discussion above \cref{prop:og_coh} and \cite[Lemma~2.1.11]{10author}.
\end{proof}

Now let $\m_Q \subset \TT^S_Q$ be the maximal ideal generated by $\m$ and the kernels of the maps 
$\OO[\til{\Xi}_{\ov{v},1}]^{S_{\mu_v}} \to k$ associated to the polynomials $(x-\alpha_v)^{d_v}, \det(X - \ov{\rho}_\m(\Frob_v))/(x-\alpha_v)^{d_v}, \det(X - \ov{\rho}_\m(\Frob_{v^c}))$ for $v \in Q$. Also let $\til{\m}_Q = {S_Q^f}^{-1}(\m_Q)$.
\begin{prop}\label{prop:descent}
    For $i=0,1$ the map $S_Q^f : \til{\TT}^S_Q \to \TT^S_Q$ descends to homomorphisms
    \[ \til{\TT}^S_Q(H^*(\til{X}^P_{\til{K}_i(Q)}, \OO)) \to \TT^S_Q(H^*(X_{K_i(Q)}, \OO))\]
    \[ \til{\TT}^S_Q(H^*(\partial \til{X}_{\til{K}_i(Q)}, \OO)_{\til{\m}}) \to \TT^S_Q(H^*(X_{K_i(Q)}, \OO)_\m)\]
\end{prop}
\begin{proof}
    To prove the first statement, we need to show that for $t \in \Ann_{\til{\TT}^S_Q}(H^*(\til{X}^P_{\til{K}_i(Q)}, \OO))$ we have $S_Q(t) \in \Ann_{\TT^S_Q}(H^*(X_{K_i(Q)}, \OO))$. Let $\alpha$ be the right inverse of $i^*$, and $\beta$ be the left inverse of $j^*$. Take any $h \in H^*(X_{K_i(Q)}, \OO)$. Then we can write 
    \begin{align*} 
        S_Q(t)h & = r_G(r_P(t))h = \beta(j^*(r_G(r_P(t))h)) = \beta(r_P(t)j^*(h)) \\ & = \beta(r_P(t)i^*(\alpha(j^*(h)))) = \beta(i^*(t \alpha(j^*(h)))) = \beta(i^*(0)) = 0.
    \end{align*}
    To prove the second statement, it is enough to note that $H^*(\til{X}^P_{\til{K}_i(Q)}, \OO)_{\til{m}} \cong H^*(\partial \til{X}_{\til{K}_i(Q)}, \OO)_{\til{\m}}$ by \cite[Theorem~2.4.2]{10author}.
\end{proof}

\section{Galois deformation theory}\label{sec:galois}

Let $E \subset \ov{\Q}_p$ be a finite extension of $\Q_p$,
with valuation ring $\OO$, uniformizer $\varpi$, and residue field $k$. Given a complete
Noetherian local $\OO$-algebra $\Lambda$ with residue field $k$, we let $\text{CNL}_\Lambda$ denote the
category of complete Noetherian local $\Lambda$-algebras with residue field $k$. We
refer to an object in $\text{CNL}_\Lambda$ as a $\text{CNL}_\Lambda$-algebra.
We fix a number field $F$, and let $S_p$ be the set of places of $F$ above $p$.
We assume that $E$ contains the images of all embeddings of $F$ in $\Q_p$. We also
fix a continuous absolutely irreducible homomorphism $\ov{\rho}: G_F \to \GL_n(k)$. We
assume throughout that $p \nmid 2n$. 

Following \cite[Definition~6.2.2]{10author} we call a global deformation problem a tuple
\[\mathcal{S} = (\ov{\rho}, S, \{\Lambda_v\}_{v \in S}, \{\mathcal{D}_v\}_{v \in S})\] where 
\begin{itemize}
\item S is a finite set of finite places of $F$ containing $S_p$ and all the places at which $\ov{\rho}$ is ramified. 
\item $\Lambda_v$ is an object of $\text{CNL}_\OO$ for each $v \in S$.
\item $\mathcal{D}_v$ is a local deformation problem (\cite[\S 6.2.1]{10author}) for each $v \in S$.
\end{itemize}
Associated to this global deformation problem we have a completed tensor product $\Lambda = \widehat{\otimes}_{v \in S}\Lambda_v$.
A global deformation problem determines a representable functor $\mathcal{D}_{\mathcal{S}} : \text{CNL}_\Lambda \to \textbf{Set}$ which takes an object $A \in \text{CNL}_\Lambda$ to the set of deformations $\rho : G_F \to \GL_n(A)$ of type $\mathcal{S}$. 

Let $v$ be a finite place of $F$ such that $v \notin S$ and $q_v \equiv 1 \pmod p$. We let $\mathcal{D}_v^1$ denote the local deformation problem consisting of all lifts which associate $A \in \text{CNL}_{\Lambda_v}$ to the set of lifts which are $1+M_n(\m_A)$-conjugate to a lift of the form $s_v \oplus \psi_v$, where $s_v$ is unramified and the image of $\psi_v$ under inertia is contained in the set of scalar matrices. This is indeed a local deformation problem, as is shown in \cite[Lemma~4.2]{Thorne12}.

\begin{lem}\label{lem:lift}
Let $\ov{r} : G_{F_v} \to \GL_n(k)$ be an unramified continuous representation, and $A$ is a complete noetherian local $\OO$-algebra with residue field $k$ and a principal maximal ideal $\m_A$. Suppose further that $\ov{r}$ may be written in the form $\ov{r} = \ov{r}_1 \oplus \ov{r}_2$, where $\det(X - \ov{r}_1(\Frob_v))$ and $\det(X - \ov{r}_2(\Frob_v))$ are relatively prime. Also suppose that $q_v = 1$ in $k$. Then any lift $r : G_{F_v} \to \GL_n(A)$ of $\ov{r}$ is $1+M_n(\m_A)$-conjugate to one of the form $r = r_1 \oplus r_2$ where $r_1$ and $r_2$ are lifts of $\ov{r}_1$ and $\ov{r}_2$ respectively.
\end{lem}
\begin{proof}
    Let $n_i = \dim \ov{r}_i$. Suppose we have a lift $r_n : G_{F_v} \to \GL_n(A)$ of $\ov{r}$ such that $r_n \bmod\;\m_A^n$ can be written in the form $r_1 \oplus r_2$. We will show that there exists a matrix $X_n \in 1+M_n(\m_A^n)$ such that $r_{n+1} \coloneqq X_nr_nX_n^{-1}$ satisfies the same condition $\bmod\;\m_A^{n+1}$. 
    Write 
    \[X_n = \begin{pmatrix}
        1 & Y \\
        Z & 1
    \end{pmatrix} \qquad r_n = \begin{pmatrix}
        A & B \\
        C & D
    \end{pmatrix}\]
    where $Y, Z \in M_n(\m_A^n)$. Then the condition on $r_{n+1}$ transforms into 
    \begin{equation}YD - AY + B = 0 \bmod\;\m_A^{n+1}\end{equation}
    \begin{equation}ZA - DZ + C = 0 \bmod\;\m_A^{n+1}\end{equation}
    We will focus on the first condition, the second is similar. We know that $r_n \mod \m_A^n$ is block-diagonal, so we can can consider $\ov{b}, \ov{y}$ to be the images of $B$ and $Y$ respectively in $\m_A^{n}/\m_A^{n+1}$. 
    \begin{equation}\label{eq:b1}\ov{b}\ov{r}_2^{-1} = \ov{r}_1\ov{y}\ov{r}_2^{-1} - \ov{y}\end{equation} in $M_n(\m_A^{n}/\m_A^{n+1}) = M_n(k) \otimes_k \m_A^{n}/\m_A^{n+1}$.
    Using the fact that $r$ is a homomorphism, for $\sigma, \tau \in G_{F_v}$ we can write 
    \[A(\sigma)B(\tau) + B(\sigma)D(\tau) = B(\sigma\tau)\]
    Rewriting and reducing $\bmod\;\m_A^{n+1}$ we get 
    \[\ov{r}_1(\sigma)\ov{b}(\tau) + \ov{b}(\sigma)\ov{r}_2(\tau) = \ov{b}(\sigma\tau)\]
    \begin{equation}\ov{b}(\sigma\tau)\ov{r}_2^{-1}(\sigma\tau) = \ov{r}_1(\sigma)\ov{b}(\tau)\ov{r}_2^{-1}(\tau)\ov{r}_2^{-1}(\sigma) + \ov{b}(\sigma)\ov{r}_2^{-1}(\sigma)\end{equation}
    Give $M_{n_1 \times n_2}(\m_A^{n}/\m_A^{n+1})$ the structure of a $G_{F_v}$-module via $\ov{r}_1(-)\ov{r}_2^{-1}$, and denote this module $\ad(\ov{r}_1, \ov{r}_2)$. Then the last equation implies that $\ov{b}\ov{r}_2^{-1}$ is in $Z^1(G_{F_v}, \ad(\ov{r}_1, \ov{r}_2))$. 
    Since $\ov{r}_1,\ov{r}_2$ have coprime characteristic polynomials we know that $H^1(G_{F_v},\ad(\ov{r}_1, \ov{r}_2)) = 0$ by local Tate duality (here we are using that $q_v = 1$ in $k$), which means\\ 
    $\ov{b}\ov{r}_2^{-1} \in B^1(G_{F_v}, \ad(\ov{r}_1, \ov{r}_2))$ and thus we can find $y$ satisfying \cref{eq:b1}.
\end{proof}

Now we define our version of the Taylor-Wiles datum, analogous to the one appearing in \cite[\S 6.2.27]{10author}.
\begin{defn}
Let \[\mathcal{S} = (\ov{\rho}, S, \{\Lambda_v\}_{v \in S}, \{\mathcal{D}_v\}_{v \in S})\] be a global deformation problem. A Taylor-Wiles datum of level $N \geq 1$ for $\mathcal{S}$ consists of a tuple $(Q, {\alpha_v}_{v \in Q})$ where 
\begin{itemize}
\item A finite set $Q$ of places of $F$, disjoint from $S$, such that $q_v \equiv 1 \pmod {p^N}$ for each $v \in Q$. 
\item For each $v \in Q$, $\alpha_v$ is an eigenvalue of $\ov{\rho}(\Frob_v)$. 
\end{itemize}
\end{defn}
Given a Taylor-Wiles datum $(Q, (\alpha_v))$ we define a global deformation problem
\[\mathcal{S}_Q = (\ov{\rho}, S \cup Q, \{\Lambda_v\}_{v \in S} \cup \{\OO_{F_v}\}_{v \in Q}, \{\mathcal{D}_v\}_{v \in S} \cup \{\mathcal{D}_v^1\}_{v \in Q})\]
 Define $\Delta_Q = \prod_{v \in Q} \Delta_v$. The representing object $R_{\mathcal{S}_Q}$ has a structure of a $\OO[\Delta_Q]$-algebra, satisfying $R_{\mathcal{S}_Q} \otimes_{\OO[\Delta_Q]} \OO = R_{\mathcal{S}}$. 
\begin{prop}\label{prop:killing_selmer}
    Take $T=S$, and let $q > h^1_{\mathcal{S}^\perp, T}(\ad \ov{\rho}(1) )$. Assume that $F = F^+F_0$ where $F_0$ is an imaginary quadratic field, that $\zeta_p \notin F$ and that $\ov{\rho}(G_{F(\zeta_p)})$ is adequate. Then for every $N \geq 1$, there exists a choice of Taylor-Wiles datum $(Q_N, (\alpha_v)_{v \in Q})$ of level $N$ satisfying the following:
    \begin{enumerate}
        \item $\lvert Q_N \rvert = q$.
        \item For each $v \in Q_N$, the rational prime below $v$ splits in $F_0$ and $v^c \notin Q_N$. 
        \item Let $g = q - n^2[F^+ : \Q]$. Then there is a surjective morphism 
        \[R^{T, loc}_{\mathcal{S}}[[X_1, \ldots, X_g]] \to R^T_{\mathcal{S}_Q}.\]
        in $\text{CNL}_\Lambda$. 
    \end{enumerate}
\end{prop}
\begin{proof}
    The proof is very similar to the proof of \cite[Proposition~6.2.32]{10author} (cf. \cite[Proopsition~4.4]{Thorne12}), we omit the details. 
\end{proof}
 
\section{Representations into Hecke algebras}
In this section we construct the necessary Galois representations into the Hecke algebras associated to $G$. From \cref{prop:descent} we know that we can create representations valued in the Hecke algebra acting on $H^*(X_{K_i(Q)}, \OO)_{\m_Q}$ from representations valued in the Hecke algebra acting on $H^*(\partial \til{X}_{\til{K}_i(Q)}, \OO)_{\til{\m}_Q}$. The latter representations will be constructed by glueing together Galois representations associated to cuspidal cohomological automorphic representations of $\til{G}(\A_{F^+}^\infty)$ as in \cite{Scholze15} and using the local computations of \cref{sec:local}.
\subsection{Hecke algebras for $\til{G}$}
\begin{thm}\label{thm:det_rel_c}
    Suppose that $\til{K} \subset \til{G}(\A_{F^+}^\infty)$ be a good subgroup which is decomposed with respect to $P$. Then there exists a $2n$-dimensional $\til{\TT}^S_Q(H_c^*(X_{\til{K}_1(Q)}, \OO))/I$-valued group determinant $D_{c,Q}$ of $G_{F, S}$ for some ideal $I$ of nilpotence degree depending only on $n$ and $[F:\Q]$ such that the following properties hold:
    \begin{enumerate}
        \item If $v \notin S$ is a place of $F$, then $D_{c, Q}(X - \Frob_v)$ is equal to the image of $\til{P}_v(X)$ in $\til{\TT}^S_Q(H_c^*(X_{\til{K}_1(Q)}, \OO))/I[X]$.
        \item If $v \notin Q$, then for any $\sigma \in G_{F, S}$ and $\tau \in I_{F_v}$ we have the relation 
            \[\Tr_{D_{c, Q}}\left(\sigma\res_{q_v, \mu_v}^{(2n)!}\res_{\mu_v}^{(2n)!}\left(\sum \limits_{i = 1}^{k-1} E_{\mu_v, i}(\varphi_v) + \langle \textnormal{Art}_{F_v}^{-1}(\tau) \rangle E_{\mu_v, k}(\varphi_v) - \res_{\mu_v}\tau \right)\right) = 0.\]
    \end{enumerate}
\end{thm}
\begin{proof}
This follows from \cref{prop:relation_1_rep} by using \cite[Theorem~2.3.3]{10author} and \cite[Corollary~5.1.11]{Scholze15} (see proof of \cite[Proposition~3.2.1]{10author}).
\end{proof}

Now we prove the version of the previous proposition for non-compactly supported cohomology:

\begin{thm}\label{thm:det_rel_nc}
    Suppose that $\til{K} \subset \til{G}(\A_{F^+}^\infty)$ be a good subgroup which is decomposed with respect to $P$. Then there exists a $2n$-dimensional $\til{\TT}^S_Q(H^*(X_{\til{K}_1(Q)}, \OO))/I$-valued group determinant $D_Q$ of $G_{F, S}$ for some ideal $I$ of nilpotence degree depending only on $n$ and $[F:\Q]$ such that the following properties hold:
    \begin{enumerate}
        \item If $v \notin S$ is a place of $F$, then $D_Q(X - \Frob_v)$ is equal to the image of $\til{P}_v(X)$ in $\til{\TT}^S_Q(H^*(X_{\til{K}_1(Q)}, \OO))/I[X]$.
        \item If $v \notin Q$, then for any $\sigma \in G_{F, S}$ and $\tau \in I_{F_v}$ we have the relation 
        \[\Tr_{D_{Q}}\left(\sigma\res_{q_v, \mu_v}^{(2n)!}\res_{\mu_v}^{(2n)!}\left(\sum \limits_{i = 1}^{k-1} E_{\mu_v, i}(\varphi_v) + \langle \textnormal{Art}_{F_v}^{-1}(\tau) \rangle E_{\mu_v, k}(\varphi_v) - \res_{\mu_v}\tau \right)\right) = 0.\]
    \end{enumerate}
\end{thm}
\begin{proof}
Denote by $\til{\TT}^S_{Q, \iota}(H^*_c(X_{\til{K}_1(Q)}, \OO))$ the image of $\til{\TT}^S_Q$ under the homomorphism 
\[\til{\TT}^S_Q \to \HH_\OO(\til{G}(\A_{F^+}^\infty), \til{K}_1(Q)) \xrightarrow{\iota_{\HH}} \HH_\OO(\til{G}(\A_{F^+}^\infty), \til{K}_1(Q)) \to \End_{\DD(\OO)}(H^*_c(X_{\til{K}_1(Q)}, \OO)).\]
The same argument as in the proof of \cref{thm:det_rel_c} shows that there exists a group determinant $D_\iota$ valued in $\til{\TT}^S_{Q, \iota}(H^*_c(X_{\til{K}_1(Q)}, \OO))/I$ satisfying the following properties:
\begin{enumerate}
    \item If $v \notin S$ is a place of $F$, then $D_Q(X - \Frob_v)$ is equal to the image of $\til{P}_v(X)$ in $\til{\TT}^S_{Q, \iota}(H^*_c(X_{\til{K}_1(Q)}, \OO))/I[X]$.
    \item If $v \notin Q$, then for any $\sigma \in G_{F, S}$ and $\tau \in I_{F_v}$ we have the relation 
    \[\Tr_{D_\iota}\left(\sigma\res_{q_v, \mu_v}^{(2n)!}\res_{\mu_v}^{(2n)!}\left(\sum \limits_{i = 1}^{k-1} E_{\mu_v, i}(\varphi_v) + \langle \textnormal{Art}_{F_v}^{-1}(\tau) \rangle E_{\mu_v, k}(\varphi_v) - \res_{\mu_v}\tau \right)\right) = 0.\]
\end{enumerate}
By \cite[Proposition~3.7]{NT} we have a commutative diagram
\newcommand*{\isoarrow}[1]{\arrow[#1,"\rotatebox{90}{\(\sim\)}"]}
\begin{equation}
\begin{tikzcd}
    \HH_\OO(\til{G}(\A_{F^+}^\infty), \til{K}_1(Q)) \arrow{r}{} \arrow{d}{\iota_{\HH}} & \End_{\DD(\OO)}(R\Gamma(X_{\til{K}_1(Q)}, \OO)) \isoarrow{d} \\ 
    \HH_\OO(\til{G}(\A_{F^+}^\infty), \til{K}_1(Q)) \arrow{r}{}  & \End_{\DD(\OO)}(R\Gamma_c(X_{\til{K}_1(Q)}, \OO))
\end{tikzcd}
\end{equation}
where the right vertical arrow is induced by Poincar\'e duality. Then we get an isomorphism \[\til{\TT}^S_{Q, \iota}(H^*_c(X_{\til{K}_1(Q)}, \OO))/I_1 \isom \til{\TT}^S_{Q}(H^*(X_{\til{K}_1(Q)}, \OO))/I_2\] over $\til{\TT}^S_Q$ for some ideals $I_{1,2}$ of nilpotence degrees depending only on $n$ and $[F:\Q]$. Moreover, we can choose $I_1$ such that it contains $I$. We can conclude by making $D_Q$ the image of $D_\iota$ under this homomorphism.
\end{proof}

\begin{lem}\label{lem:Burnside_2}
    Let $k$ be a field, and let $\ov{\rho}_1, \ov{\rho}_2: G \to GL_n(k)$ be two non-isomorphic absolutely irreducible representations. Then the extended map $k[G] \to M_n(k) \oplus M_n(k)$ defined by $\ov{\rho}_1 \oplus \ov{\rho}_2$ is surjective.
\end{lem}
\begin{proof}
    We may pass to the algebraic closure of $k$ (which we still denote $k$). Let $\ell_i : k[G] \to M_n(k)$ be the linear extension of $\ov{\rho}_i$ for $i = 1,2$. The two maps $\ell_i$ are surjective by Burnside's theorem. Let $A$ be the image of $\ell_1 \oplus \ell_2$, and let $I_i = \ker(A \to M_n(k))$ where $i=1,2$ corresponds to projecting on the first and second factor. Since $\ell_i$ are surjective, $I_i$ are in fact two-sided ideals of $M_n(k)$. Then $I_i = M_n(k)$ or $I_i = 0$. If $I_i = M_n(k)$ for some $i$, then $\ell_1 \oplus \ell_2$ is surjective. Suppose then that $I_1 = I_2 = 0$. Then we have an automorphism $f$ of $M_n(k)$ defined by $(v, f(v)) \in A$ for all $v \in M_n(k)$. Since all the automorphisms of $M_n(k)$ are inner, we conclude that there exists $u \in \GL_n(k)$ such that $A = \{(v, uvu^{-1}) \mid v \in M_n(k)\}$. But this is impossible since $\ov{\rho}_1$ and $\ov{\rho}_2$ are non-isomorphic.
\end{proof}

\begin{thm}\label{thm:gtil_main}
    Suppose that $\til{K} \subset \til{G}(\A_{F^+}^\infty)$ be a good subgroup which is decomposed with respect to $P$ and that for each $v \in Q$ we have $\res_{\mu_v} \notin \til{\m}_{Q}$. Then there exists a continuous representation 
    \[\rho_{\m_Q} : G_{F, S \cup Q} \to \GL_n(\TT_Q^S(H^*(X_{K_1(Q)}, \OO)_{\m_Q})/I)\]
    satisfying the conditions below for some ideal $I \subset \TT_Q^S(H^*(X_{K_1(Q)}, \OO)_{\m_Q})$ of nilpotence degree depending only on $n$ and $[F : \Q]$.
    \begin{enumerate}
        \item If $v \notin S$ is a place of $F$, the characteristic polynomial of $\rho_{\m_Q}(\Frob_v)$ is equal to the image of $P_v(X)$ in $\TT_Q^S(H^*(X_{K_1(Q)}, \OO)_{\m})/I[X]$.
        \item If $v \in Q$, then $\rho_{\m_Q}|_{G_{F_{v^c}}}$ is unramified. 
        \item If $v \in Q$, then $\rho_{\m_Q}|_{G_{F_v}} = s \oplus \psi$ where $s$ is unramified and $\tau \in I_{F_v}$ acts on $\psi$ as a scalar $\langle \textnormal{Art}_{F_v}^{-1}(\tau) \rangle$. 
    \end{enumerate}
\end{thm}
\begin{proof}
    Using \cref{thm:det_rel_c} and \cref{thm:det_rel_nc}, we can construct a\\ $\til{\TT}^S_Q(H_c^*(X_{\til{K}_1(Q)}, \OO)_{\til{\m}_Q} \oplus H^*(X_{\til{K}_1(Q)}, \OO)_{\til{\m}_Q})/I$-valued group determinant $D_Q$ of $G_{F,S \cup Q}$.
    Consider the long exact sequence 
    \[\ldots \to H_c^i(\til{X}_{\til{K}_1(Q)}, \OO) \to H^i(\til{X}_{\til{K}_1(Q)}, \OO) \to H^i(\partial \til{X}_{\til{K}_1(Q)}, \OO) \to H_c^{i+1}(\til{X}_{\til{K}_1(Q)}, \OO)\to \]
    Using this sequence and \cref{prop:descent} we know that $S_Q^f$ descends to a homomorphism 
    \[\til{\TT}^S_Q(H_c^*(\til{X}_{\til{K}_1(Q)}, \OO)_{\til{\m}_Q} \oplus H^*(\til{X}_{\til{K}_1(Q)}, \OO)_{\til{\m}_Q}) \to \TT^S_Q(H^*(X_{K_1(Q), \OO})_{\m_Q})/I_0\]
    for some ideal $I_0$ with square $0$. We can use this to construct a $2n$-dimensional group determinant $D^0_Q$ valued in $\TT^S_Q(H^*(X_{K_1(Q), \OO})_{\m_Q})/I$ such that:
    \begin{enumerate}
        \item For $v \notin S$ we have $D^0_Q(X - \Frob_v) = P_v(X)q_v^{n(2n-1)}{P_{v^c}}^\vee (q_v^{1-2n}X)$
        \item For $v \in Q$ we have 
        \[\Tr_{D^0_Q}\left(S_Q^f\left(\sigma\res_{q_v, \mu_v}^{(2n)!}\res_{\mu_v}^{(2n)!}\left(\sum \limits_{i = 1}^{k-1} E_{\mu_v, i}(\varphi_v) + \langle \textnormal{Art}_{F_v}^{-1}(\tau) \rangle E_{\mu_v, k}(\varphi_v) - \res_{\mu_v}\tau \right)\right)\right) = 0.\]
    \end{enumerate}
    and $I$ has nilpotence degree depending only on $n$ and $[F:\Q]$. By \cite[Theorem~2.3.7]{10author} there also exists an $n$-dimensional group determinant $D^1_Q$ of $G_{F, S \cup Q}$ valued in $\TT^S_Q(H^*(X_{K_1(Q), \OO})_{\m_Q})/I$ such that $D^1_Q(X - \Frob_v) = P_v(X)$ for $v \notin S$. Then the group determinants $D^1_Q \oplus {D^1_Q}^\perp$ and $D^0_Q$ are equal. Moreover, since $\ov{\rho}_\m$ is absolutely irreducible, there exists a continuous representation 
    \[\rho_{\m_Q} : G_{F, S \cup Q} \to \GL_n(\TT^S_Q(H^*(X_{K_1(Q)}, \OO)_{\m_Q})/I)\]
    such that the characteristic polynomial of $\rho_{\m_Q}$ is associated to $D^1_Q$. Let $\rho_{\m_Q}' \coloneqq \rho_{\m_Q}\oplus \rho_{\m_Q}^\perp$. Writing out the relation at places $v \in Q$, we get 
    \begin{align*}& \Tr (\rho_{\m_Q}'(\sigma) S_Q^f(\res_{q_v, \mu_v}^{(2n)!}\res_{\mu_v}^{(2n)!}(\sum \limits_{i = 1}^{k-1} E_{\mu_v, i}(\rho_{\m_Q}'(\varphi_v)) \\ & + \langle \textnormal{Art}_{F_v}^{-1}(\tau) \rangle E_{\mu_v, k}(\rho_{\m_Q}'(\varphi_v)) - \res_{\mu_v}\rho_{\m_Q}'(\tau) ))) =  0.\end{align*}
    Since $\res_{\mu_v} \notin \til{\m}_{Q}$, we know that $\ov{\rho}_{\m}$ and $\ov{\rho}_{\m}^\perp$ are not isomorphic. Applying Nakayama's lemma and \cref{lem:Burnside_2} we see that the extended map 
    \[\TT^S_Q[G_{F, S \cup Q}] \to M_n(\TT^S_Q(H^*(X_{K_1(Q), \OO})_{\m_Q})/I) \oplus M_n(\TT^S_Q(H^*(X_{K_1(Q), \OO})_{\m_Q})/I)\] given by $\rho_{\m_Q}\oplus \rho_{\m_Q}^\perp$ is surjective. Considering the trace relation above with $\sigma$ replaced by an arbitrary element of $\TT^S_Q[G_{F, S \cup Q}]$ we conclude that 
    \begin{align*}& S_Q^f(\res_{q_v, \mu_v}^{(2n)!}\res_{\mu_v}^{(2n)!}(\sum \limits_{i = 1}^{k-1} E_{\mu_v, i}(\rho_{\m_Q}'(\varphi_v)) \\ & + \langle \textnormal{Art}_{F_v}^{-1}(\tau) \rangle E_{\mu_v, k}(\rho_{\m_Q}'(\varphi_v)) - \res_{\mu_v}\rho_{\m_Q}'(\tau) )) =  0.\end{align*}
    
    Since $q_v \equiv 1 \mod p$, we know that $\res_{q_v, \mu_v} \notin \til{\m}_Q$. Thus 
    \begin{align*}\label{eq:fin_rel} S_Q^f\left(\sum \limits_{i =  1}^{k-1} E_{\mu_v, i}(\rho_{\m_Q}'(\varphi_v)) + \langle \textnormal{Art}_{F_v}^{-1}(\tau) \rangle E_{\mu_v, k}(\rho_{\m_Q}'(\varphi_v)) - \res_{\mu_v}\rho_{\m_Q}'(\tau) \right) =  0.\end{align*}
    This implies that
    \[\rho_{\m_Q}(\tau) = S_Q^f\left(\sum \limits_{i = 1}^{k-1} \res_{\mu_v}^{-1}E_{\mu_v, i}(\rho_{\m_Q}(\varphi_v))\right)
     + S_Q^f(\langle \textnormal{Art}_{F_v}^{-1}(\tau)\rangle \res_{\mu_v}^{-1}E_{\mu_v, k}(\rho_{\m_Q}(\varphi_v)))\]
     \[\rho_{\m_Q}(\tau) = \res_{\nu_v}^{-1}E_{\nu, 1}(\rho_{\m_Q}(\varphi_v))
     + \langle \textnormal{Art}_{F_v}^{-1}(\tau)\rangle \res_{\mu_v}^{-1}E_{\nu, 2}(\rho_{\m_Q}(\varphi_v))\]
     Let $\TT \coloneqq \TT^S_Q(H^*(X_{K_1(Q), \OO})_{\m_Q})/I$. 
    Consider the decomposition $\ov{\rho}_\m = \ov{r}_1 \oplus \ov{r}_2$, corresponding to the Frobenius generalized eigenspaces of all eigenvalues not equal to $\alpha_v$ and $\alpha_v$ respectively. Then \[\TT^n = \res_{\nu_v}^{-1}E_{\nu, 1}(\rho_{\m_Q}(\varphi_v))\TT^n
    \oplus \res_{\mu_v}^{-1}E_{\nu, 2}(\rho_{\m_Q}(\varphi_v))\TT^n\]
    is  the unique $\rho_{\m_Q}(\varphi_v)$-invariant lift of $\ov{r}_1 \oplus \ov{r}_2$, and we are done by \cref{lem:lift}.

\end{proof} \subsection{Hecke algebras for $G$}
Let $\lambda \in (\Z_+^n)^{\Hom(F, E)}$. Further let $S$ be a finite set of finite places of $F$ containing the $p$-adic places and stable under complex conjugation, satisfying the following conditions:
\begin{enumerate}
    \item Let $l$ be a rational prime such that there exists a place above $l$ in $S$ or $l$ is ramified in $F$. Then there exists an imaginary quadratic subfield $F_0 \subset F$ such that $l$ splits in $F_0$. 
    \item If $v \in S$ then $v$ is split in $F^+$.
\end{enumerate}
Let $K \subset \GL_n(\A_F^\infty)$ be a good subgroup such that for all $v \notin S$ we have $K_v = \GL_n(\OO_{F_v})$. Let $\m \subset \TT^S(K, \lambda)$ be a non-Eisenstein maximal ideal with residue field $k$. By \cite[Theorem~2.3.5]{10author} there exists an associated residual representation $\ov{\rho}_\m : G_{F, S} \to \GL_n(\TT^S(K, \lambda)/{\m})$. By \cite[Theorem~2.3.7]{10author} there exists an ideal $I \subset \TT^S(K, \lambda)$ of nilpotence degree depending only on $n$ and $[F : \Q]$ and a continuous lift $\rho_\m : G_{F, S} \to \GL_n(\TT^S(K, \lambda)_{\m}/I)$ such that for each $v \in S$, $\det(X - \rho_\m(\Frob_v))$ is the image of $P_v(X)$ in $\TT^S(K, \lambda)_{\m}/I[X]$.
We consider the following Taylor-Wiles datum: a tuple $(Q, (\alpha_v)_{v \in Q})$ consisting of 
\begin{itemize}
\item A finite set $Q$ of places of $F$, disjoint from $Q^c$, such that $q_v \equiv 1 \pmod p$ for each $v \in Q$. 
\item Each $v \in Q$ is split in $F^+$, and there exists an imaginary quadratic subfield $F_0 \subset F$ such that $v$ is split in $F_0$. Moreover, $\ov{\rho}_\m$ is unramified at $v$ and $v^c$.
\item $\alpha_v$ is a root of $\det(X - \ov{\rho}_\m(\Frob_v))$.
\end{itemize}
Consider the partition $\nu_v : n = d_v + (n-d_v)$, where $d_v$ is the multiplicity of $\alpha_v$ as a root of $\det(X - \ov{\rho}_\m(\Frob_v))$. 

We define auxillary level subgroups $K_1(Q) \subset K_0(Q) \subset K$. They are good subgroups of $\GL_n)(\A_F^\infty)$
define by the following conditions:
\begin{itemize}
    \item if $v \notin Q$, then $K_1(Q)_v = K_0(Q)_v = K_v$.
    \item if $v \in Q$, then $K_0(Q)_v = \p_{\nu_v}$ and $K_1(Q)_v = \p_{\nu_v,1}$.
\end{itemize}
We have a natural isomorphism $K_0(Q)/K_1(Q) \cong \Delta_Q = \prod_{v \in Q} \Delta_v$. 
Let $S' = S \cup Q \cup Q^c$. We define $\TT^{S'}_Q = \TT^{S \cup Q} \otimes_{\Z} \Z[\Xi_{v, 1}]^{S_{\nu_v}}$.
Let $\TT^{S'}_Q(K_0(Q), \lambda)$ and $\TT^{S'}_Q(K_0(Q)/K_1(Q), \lambda)$ be the images of $\TT^{S'}_Q$ in $\End_{\DD(\OO)}(R\Gamma(X_{K_0(Q)}, V_\lambda))$ and\\ $\End_{\DD(\OO[\Delta_Q])}(R\Gamma(X_{K_1(Q)}, V_\lambda))$ respectively. Let $\m_Q$ be the maximal ideal of $\TT^{S'}_Q$ generated by $\m$ and the kernels of the homomorphisms 
$\Z[\Xi_{v, 1}]^{S_{\nu_v}} \to k$ given by the coefficients of polynomials $(X - \alpha_v)^{d_v}, \det(X - \ov{\rho}_\m(\Frob_v))/(X - \alpha_v)^{d_v}$ . 
\begin{thm}\label{thm:cusp_growth}
    The natural morphisms 
    \[R\Gamma(X_{K_0(Q)}, V_\lambda)_{\m_Q} \to R\Gamma(X_K, V_\lambda)_{\m}\]
    \[R\Gamma(\Delta_Q, R\Gamma(X_{K_1(Q)}, V_\lambda))_{\m_Q} \to R\Gamma(X_{K_0(Q)}, V_\lambda)_{\m_Q}\]
    in $\DD(\OO)$ are isomorphisms.
\end{thm}
\begin{proof}
The second isomorphism is straightforward. For the first, we can check on the level of cohomology.
    It is enough to check that it is an isomorphism in $\DD(k)$ after applying the functor $-\otimes^{\mathbf{L}}k$. 
    Thus we need to show that the map
    \[H^*(X_{K_0(Q)}, V_\lambda/\varpi)_{\m_Q} \to H^*(X_{K}, V_\lambda/\varpi)_{\m}\]
    is an isomorphism. We can prove this one prime at a time, so we can assume $Q = \{v\}$. 
    For each $j$, let \[M_j \coloneqq \lim_{m \to \infty} H^j(X_{K(v^m)}, V_\lambda/\varpi)_{\m}\]
    where $K(v^m)_w = K_w$ for places $w \neq v$ and $K(v^m)_v$ ranges over all congruence subgroups of level $v^m$.
    We have two Hochschild-Serre spectral sequences:
    \[H^i(\GL_n(\OO_{F_v}), M_j) \Rightarrow H^{i+j}(X_K, V_\lambda/\varpi)_{\m}\]
    \[e_{\alpha_v}H^i(\p_{\nu_v}, M_j) \Rightarrow e_{\alpha_v}H^{i+j}(X_{K_0(Q)}, V_\lambda/\varpi) = H^{i+j}(X_{K_0(Q)}, V_\lambda/\varpi)_{\m_Q}\]
    There is a natural map $\iota^*$ between the these spectral sequences, which arises from deriving the map 
    \[M_j^{\GL_n(\OO_{F_v})} \to M_j^{\p_{\nu_v}} \to e_{\alpha_v}M_j^{\p_{\nu_v}}\]
    Thus it is enough to show that $\iota^*$ map is an isomorphism. $M_j$ is admissible, and we can use 
    \cite[III.6]{VigInd} to write $M_j$ as a direct sum of $\GL_n(F_v)$-modules, each belonging to a single block.
    Let $N \subset M_j$ be a summand from a non-unipotent block. We note that both $H^i(\GL_n(\OO_{F_v}), N)$ and
    $H^i(\p_{\nu_v}, N)$ inject into $H^i(T^p(k), N^{\Iw^p})$. Since $N$ is a from a non-unipotent block we know that 
    $N^{\Iw^p} = 0$, and so \[H^i(\GL_n(\OO_{F_v}), N) = H^i(\p_{\nu_v}, N) = 0\]
    Thus we can restrict to the summand $M_j^1 \subset M_j$ from the unipotent block, and it is enough to prove that 
    \[\iota^* : H^i(\GL_n(\OO_{F_v}), M_j^1) \to e_{\alpha_v}H^i(\p_{\nu_v}, M_j^1)\]
    is an isomorphism. By \cite[Theorem~B.1]{CHT} the unipotent block in our case consists of representations generated by their $\Iw^p$-invariant vectors. 
    Therefore every irreducible subrepresentation $\pi \subset M_j^1$ has a $\Iw^p$-invariant vector. It follows from the argument similar to the proof of \cref{prop:embeds}
    that \[\pi \subset \ind_B^{GL_n} \chi_1\otimes \ldots \otimes \chi_n\] 
    where $\chi_i$ are tamely ramified characters whose restriction to $\OO_{F_v}/(1+\varpi\OO_{F_v})$ has $p$-power order. 
    But these characters are valued in $k^\times$ which has order coprime to $p$, which means $\chi_i$ are in fact unramified.
    
    We can now select a smallest number $d > 0$ such that $\pi$ embeds into $M_j[\m^d]$. Since $\pi$ is irreducible
    it must then embed into $M_j[\m^d]/M_j[\m^{d-1}]$, and local-global compatibility for Iwahori level (\cite[Theorem~3.1.1]{10author}) then implies that\\
    $\{\chi_i(\varpi)\}_{i = 1, \ldots, n}$ is the set of eigenvalues of $\ov{\rho}_{\m}(\Frob_v)$. Thus we have shown that 
    $M_j \in \CC$, and we are done by \cref{thm:c_iso}.
\end{proof}

\begin{thm}\label{thm:main_g}
    There exists an ideal $I \subset \TT^{S'}_Q(K_0(Q)/K_1(Q), \lambda)_{\m_Q}$ of nilpotence degree depending only on $n$ and $[F:\Q]$, together with a continuous homomorphism 
    \[\rho_{\m, Q} : G_{F, S \cup Q} \to \GL_n(\TT^{S'}_Q(K_0(Q)/K_1(Q), \lambda)_{\m_Q}/I)\]
    lifting $\ov{\rho}_{\m}$ and satisfying the following conditions:
    \begin{enumerate}
        \item For a finite place $v \notin S \cup Q$ of $F$, $\det(X - \rho_{\m, Q}(\Frob_v))$ equals to the image of $P_v(X)$ in $\TT^{S'}_Q(K_0(Q)/K_1(Q), \lambda)_{\m_Q}/I[X]$. 
        \item For $v \in Q$, $\rho_{\m, Q}|_{G_{F_{v^c}}}$ is unramified and $\rho_{\m, Q}|_{G_{F_{v}}}$ is a lifting of type $\D_v$, and the induced map $\OO[\Delta_Q] \to \TT^{S'}_Q(K_0(Q)/K_1(Q), \lambda)_{\m_Q}/I$ is a homomorphism of $\OO[\Delta_Q]$-algebras. 
    \end{enumerate}
\end{thm}
\begin{proof}
    We first make a few reductions. Let us show that we can reduce to the situation where $\det(X - \ov{\rho}_\m(\Frob_v))$ and $\det(X - \ov{\rho}_\m(\Frob_{v^c}))$ are coprime for each $v \in Q$. To achieve this we will use twisting. Pick an odd prime $l \neq p$, and consider a character $\psi : G_F \to \OO^\times$ of order $\ell$ such that $\det(X - (\ov{\rho}_\m \otimes \ov{\psi})(\Frob_v))$ and $\det(X - (\ov{\rho}_\m \otimes \ov{\psi})(\Frob_{v^c}))$ are coprime. Let $S_{\psi}$ denote the places of $F$ at which $\psi$ is ramified. We will further require that $S_{\psi}$ is disjoint from $S'$. Define a good subgroup $K^{\psi} \subset K$ given by $K^{\psi}_v = K_v$ at places $v$ at which $\psi$ is unramified, and $K^{\psi}_v = \ker(\GL_n(\OO_{F_v}) \to k(v)^\times/(k(v)^\times)^l)$ at places $v$ where $\psi$ is ramified. Following the discussion above \cite[Proposition~2.2.14]{10author} we have a homomorphism $f_{\psi} : \TT^{S' \cup S_\psi}(K^\psi, \lambda) \to \TT^{S' \cup S_\psi}(K^\psi, \lambda)$ given by 
    \begin{equation}\label{eq:twist}f_\psi([{K^\psi}^{S' \cup S_\psi} g {K^\psi}^{S' \cup S_\psi}]) = \psi^{-1}(\text{Art}(\det(g)))[{K^\psi}^{S' \cup S_\psi} g {K^\psi}^{S' \cup S_\psi}].\end{equation}
    We have a maximal ideal $\m_{\psi} = f_\psi(\m)$ of $\TT^{S' \cup S_\psi}(K^\psi, \lambda)$. \cite[Proposition~2.2.14]{10author} implies an isomorphism $\ov{\rho}_{\m} \otimes \ov{\psi} \cong \ov{\rho}_{\m_\psi}$. 
    Similarly to \cref{eq:twist}, we have an isomorphism 
    \[\TT^{S' \cup S_\psi}_Q(K^\psi_0(Q)/K^\psi_1(Q), \lambda)_{{\m_\psi}_Q} \cong \TT^{S' \cup S_\psi}_Q(K^\psi_0(Q)/K^\psi_1(Q), \lambda)_{{\m}_Q}\]
    where ${\m_\psi}_Q$ is the maximal ideal of $\TT^{S'\cup S_\psi}_Q$ generated by $\m_\psi$ and the kernels of the homomorphisms $\Z[\Xi_{v, 1}]^{S_{\nu_v}} \to k$ given by the coefficients of polynomials $(X - \psi(\Frob_v)\alpha_v)^{d_v}, \det(X - \ov{\rho}_{\m_\psi}(\Frob_v))/(X - \psi(\Frob_v)\alpha_v)^{d_v}$. 
    We have a surjective map of $\TT^{S' \cup S_\psi}$-algebras
    \[\TT^{S' \cup S_\psi}_Q(K^\psi_0(Q)/K^\psi_1(Q), \lambda)_{{\m}_Q} \to \TT^{S' \cup S_\psi}_Q(K_0(Q)/K_1(Q), \lambda)_{{\m}_Q}.\]
    Thus if the theorem holds for representations into $\TT^{S' \cup S_\psi}_Q(K^\psi_0(Q)/K^\psi_1(Q), \lambda)_{{\m}_Q}$ it will hold for representations into $\TT^{S' \cup S_\psi}_Q(K_0(Q)/K_1(Q), \lambda)_{{\m}_Q}$. Since there are infinitely many $\psi$ satisfying the conditions we require, we can vary them to conclude that the theorem holds for $\TT^{S'}_Q(K_0(Q)/K_1(Q), \lambda)_{{\m}_Q}$, which is our target Hecke algebra. 

    Let $\til{K} \subset \til{G}(\A_{F^+}^\infty)$ be a good subgroup satisfying the following conditions:
    \begin{enumerate}
        \item $K$ is decomposed with respect to $P$.
        \item if $\ov{v}$ is a finite place of $F^+$ such that $\ov{v} \notin \ov{S}$, then $\til{K}_{\ov{v}} = \til{G}(\OO_{F^+_{\ov{v}}})$.
    \end{enumerate}
    We can use the Hochschild-Serre spectral sequence to reduce to the case where $K = \til{K} \cap G(\A_{F^+}^\infty)$. We can further reduce our theorem to the case $\lambda = 0$, by a standard use of the Hochschild-Serre spectral sequence to trivialize the weight modulo some power $m$ at the expense of shrinking the level at $p$.
    Now the theorem follows from \cref{thm:gtil_main}.
\end{proof}

\section{Proof of \cref{thm:main1} and \cref{thm:main2}}

Let us recall the proof structure of \cite[Theorem~6.1.1]{10author}. The theorem is reduced in \cite{10author} to \cite[Corollary~6.5.5]{10author}, which is prove using \cite[Theorem~6.5.4]{10author}. The reduction does not use the `enormous' assumption on the image of $\ov{\rho}$. Thus it will be sufficient for us to prove an analog of  \cite[Theorem~6.5.4]{10author}, replacing `enormous' by `adequate' in the hypotheses.

Let $F$ be an imaginary CM number field, and fix the following data:
\begin{enumerate}
\item An integer $n > 2$ and a prime $p > n^2$.
\item A finite set $S$ of finite places of $F$, including the places above $p$.
\item A (possibly empty) subset $R \subset S$ of places which are prime to $p$.
\item A cuspidal automorphic representation $\pi$ of $\GL_n(\A_F)$, which is regular algebraic of some weight $\lambda$.
\item A choice of isomorphism $\iota : \ov{\Q}_p \cong \C$.

We assume that the following conditions are satisfied:\\

\item If $l$ is a prime lying below an element of $S$, or which is ramified in $F$, then $F$ contains an imaginary quadratic field in which $l$ splits. In particular, each place of $S$ is split over $F^+$ and the extension $F/F^+$ is everywhere unramified. 
\item The prime $p$ is unramified in $F$. 
\item For each embedding $\tau : F \xhookrightarrow{} \C$, we have 
\[\lambda_{\tau,1} + \lambda_{\tau c,1} - \lambda_{\tau,n} - \lambda_{\tau c,n} < p-2n.\]
\item For each $v \in S_p$, let $\ov{v}$ denote the place of $F^+$ lying below $v$. Then there exists a place $\ov{v}' \neq \ov{v}$ of $F^+$ such that $\ov{v}' \mid p$ and 
\[\sum \limits_{\ov{v}'' \neq \ov{v}, \ov{v}'} [F_{\ov{v}''}^+ : \Q_p] > \frac{1}{2}[F^+ : \Q].\]
\item The residual representation $\ov{r_\iota(\pi)}$ is absolutely irreducible.
\item If $v$ is a place of $F$ lying above $p$, then $\pi_v$ is unramified. 
\item if $v \in R$, then $\pi_v^{\Iw_v} \neq 0$.
\item If $v \in S - (R \cup S_p)$, then $\pi_v$ is unramified and $H^2(F_v, \ad \ov{r_\iota(\pi)}) = 0$.\\
Moreover, $v$ is absolutely unramified and of residue characteristic $q > 2$.
\item $S - (R \cup S_p)$ is nonempty.
\item If $v \notin S$ is a finite place of $F$, then $\pi_v$ is unramified. 
\item If $v \int R$, then $q_v \equiv 1 \pmod p$ and $\ov{r_\iota(\pi)}|_{G_{F_v}}$ is trivial. 
\item The representation $\ov{r_\iota(\pi)}$ is decomposed generic in the sense of \cite[Definition~4.3.1]{10author} and the image of $\ov{r_\iota(\pi)}|_{G_{F(\zeta_p)}}$ is adequate.
\end{enumerate}

We define an open compact subgroup $K = \prod_v K_v$ of $\GL_n(\widehat{\OO}_F)$ as follows:
\begin{itemize}
\item If $v \notin S$, or $v \in S_p$, then $K_v = \GL_n(\OO_{F_v})$. 
\item If $v \in R$, then $K_v = \Iw_v$.
\item If $v \in S - (R \cup S_p)$, then $K_v$ is the principal congruence subgroup of $\GL_n(\OO_{F_v})$.
\end{itemize}

By \cite[Theorem~2.4.9]{10author}, we can find a coefficient field $E \subset \ov{\Q}_p$ and a maximal ideal $\m \subset \TT^S(K, \V_\lambda)$ such that $\ov{\rho}_{\m} \cong \ov{r_{\iota}(\pi)}$. After possibly enlarging $E$, we can and do assume that the residue field of $\m$ is equal to $k$. For each tuple $(\chi_{v,i})_{v \in R, i=1,\ldots,n}$ of characters $\chi_{v,i} : k(v)^\times \to \OO^\times$ which are trivial modulo $\varpi$, we define a global deformation problem by the formula 
\[S_\chi = (\ov{\rho}_\m, S, \{\OO\}_{v \in S}, \{\mathcal{D}_v^{\text{FL}}\}_{v \in S_p} \cup \{\mathcal{D}_v^\chi\}_{v \in R}) \cup \{\mathcal{D}_v^\square\}_{v \in S - (R \cup S_p)}).\]
We fix representatives $\rho_{S_\chi}$ of the universal deformations which are identified modulo $\varpi$ via the identifications $R_{S_\chi}/\varpi \cong R_{S_1}/\varpi$. We define an $\OO[K_S]$-module $\V_\lambda(\chi^{-1}) = \V_\lambda \otimes_{\OO} \OO(\chi^{-1})$, where $K_S$ acts on $V_\lambda$ by projection to $K_p$ and on $\OO(\chi^{-1})$ by the projection $K_S \to K_R = \prod_{v \in R} \Iw_v \to \prod_{v \in R} (k(v)^\times)^n$.

\begin{thm}\label{thm:nearly_faithful}
    Under assumptions (1)-(17) above, $H^*(X_K, \V_\lambda(1))_{\m}$ is a nearly faithful $R_{S_1}$-module. In other words, $\Ann_{R_{S_1}}(H^*(X_K, \V_\lambda(1))_{\m})$ is nilpotent. 
\end{thm}
The rest of the paper is devoted to the proof of \cref{thm:nearly_faithful}.

Consider the Taylor-Wiles datum $(Q, \{\alpha_v\}_{v \in Q})$ satisfying the following conditions:
\begin{itemize}
    \item For each place $v \in Q$ of residue characteristic $l$ there exists an imaginary quadratic subfield $F_0 \subset F$ such that $l$ splits in $F_0$.
    \item $Q$ and $Q^c$ are disjoint.
\end{itemize}
We have the following result, combining \cite[Proposition~6.5.3]{10author} and \cref{thm:main_g}: 
\begin{prop}\label{prop:RtoT_2}
    There exists an integer $\delta \geq 1$ depending only on $n$ and $[F:\Q]$, an ideal $J \subset \TT^{S'}_Q(R\Gamma(X_{K_1(Q)}, V_\lambda(\chi^{-1}))_{\m_Q})$ such that $J^\delta = 0$ and a continuous surjection of $\OO[\Delta_Q]$-algebras $f_{S_{\chi, Q}} : R_{\chi, Q} \to \TT^{S'}_Q(R\Gamma(X_{K_1(Q)}, V_\lambda(\chi^{-1}))_{\m_Q})/J$ such that for each finite place $v \notin S \cup Q$ the characteristic polynomial of $f_{S_{\chi, Q}} \circ \rho_{S_{\chi, Q}}$ equals the image of $P_v(X)$. 
\end{prop}

Let \[q = h^1(F_S/F, \ad \ov{\rho}_\m(1)) \quad \text{and} \quad g = q - n^2[F^+:\Q],\] 
and set $\Delta_\infty = \Z_p^{q}$. Let $\mathcal{T}$ be a power series ring over $\OO$ in $n^2\lvert S \rvert - 1$ variables, and let $S_\infty = \mathcal{T}[[\Delta_\infty]]$. Let $\mathfrak{a}_\infty$ be the augmentation ideal of $S_\infty$ viewed as an augmented $\OO$-algebra. Since $p > n$, for each $v \in R$ we can choose a tuple of pairwise distinct characters $\chi_v = (\chi_{v,1}, \ldots, \chi_{v,n})$, with $\chi_{v,i} : \OO_{F_v}^\times \to \OO^\times$ trivial modulo $\varpi$. We write $\chi$ for the tuple $(\chi_v)_{v \in R}$ as well as for the induced character $\prod_{v \in R}I_v \to \OO^\times$. Fix an imaginary quadratic subfield $F_0 \subset F$. Then for each $N \geq 1$,  we fix a choice of  Taylor-Wiles datum $(Q, \{\alpha_v\}_{v \in Q})$ for $\mathcal{S}_1$ of level $N$ using \cref{prop:killing_selmer}. For $N = 0$ we set $Q_0 = \emptyset$. 
For each $N \geq 1$, we set $\Delta_N = \Delta_{Q_N}$, and fix a surjection $\Delta_\infty \to \Delta_N$. We let $\Delta_0$ be the trivial group, viewed as a quotient of $\Delta_\infty$. 
For each $N \geq 0$, we set $R_N = R_{\mathcal{S}_1, Q_N}$ and $R'_N = R_{\mathcal{S}_\chi, Q_N}$. Let $R^{loc} = R_{\mathcal{S}_1}^{S, loc}$ and ${R'}^{loc} = {R'}_{\mathcal{S}_\chi}^{S, loc}$ denote the local deformation rings.
We let $R_\infty$ and $R'_\infty$ be formal power series rings in $g$ variables over $R^{loc}$ and ${R'}^{loc}$, respectively. We also have canonical isomorphisms $R_N/\varpi \cong R'_N/\varpi$ and $R^{loc}/\varpi \cong {R'}^{loc}/\varpi$. Using \cite[Proposition~6.2.24]{10author} and \cite[Proposition~6.2.31]{10author}, we have local $\OO$-algebra surjections $R_\infty \to R_N$ and $R'_\infty \to R'_N$ for $N \geq 0$. We can and do assume that these are compatible with the fixed identifications modulo $\varpi$, and with the isomorphisms $R_N \otimes_{\OO[\Delta_Q]} \OO = R_0$ and $R'_N \otimes_{\OO[\Delta_Q]} \OO = R'_0$. 

Define $\mathcal{C}_0 = R\Hom_\OO(R\Gamma(X_{K}, V_\lambda(1))_\m, \OO)[-d] \in \DD(\OO)$ and $T_0 = \TT^S(\mathcal{C}_0)$. Similarly, we define $\mathcal{C}'_0 = R\Hom_\OO(R\Gamma(X_{K}, V_\lambda(\chi^{-1}))_\m$ and $T'_0 = \TT^S(\mathcal{C}'_0)$. For any $N \geq 1$, we let 
\[\mathcal{C}_N = R\Hom_\OO(R\Gamma(X_{K_1(Q)}, V_\lambda(1))_{\m_{Q_N}}, \OO)[-d],\]
and 
\[T_N = \TT^{S'}_Q(\mathcal{C}_N)\]
Similarly, we let
\[\mathcal{C}'_N = R\Hom_\OO(R\Gamma(X_{K_1(Q)}, V_\lambda(\chi^{-1}))_{\m_{Q_N}}, \OO)[-d]\]
and 
\[T'_N = \TT^{S'}_Q(\mathcal{C}'_N).\]

For any $N \geq 0$, there are canonical isomorphisms \[\mathcal{C}_N \otimes^{\mathbf{L}}_{\OO[\Delta_N]} k[\Delta_N] \cong \mathcal{C}'_N \otimes^{\mathbf{L}}_{\OO[\Delta_N]} k[\Delta_N] \]
in $\DD(k[\Delta_N])$. These yield the identification 
\[\End_{\DD(\OO)}(\mathcal{C}_N \otimes^{\mathbf{L}}_{\OO} k) \cong \End_{\DD(\OO)}(\mathcal{C}'_N \otimes^{\mathbf{L}}_{\OO} k).\]
Thus we can write $\ov{T}_N$ for the image of both $T_N$ and $T'_N$ in the identified endomorphism algebras. By \cref{thm:cusp_growth}, there are canonical isomorphisms $\mathcal{C}_N \otimes^{\mathbf{L}}_{\OO[\Delta_N]} \OO \cong \mathcal{C}_0$ and $\mathcal{C}'_N \otimes^{\mathbf{L}}_{\OO[\Delta_N]} \OO \cong \mathcal{C}'_0$ in $\DD(\OO)$, which are compatible with the reductions modulo $\varpi$. By \cref{prop:RtoT_2} we can find and integer $\delta \geq 1$ and for each $N \geq 0$ ideals $I_N$ of $T_N$ and $I'_N$ of $T'_N$ of nilpotence degree $\leq \delta$ such that there exist local $\OO[\Delta_N]$-algebra surjections $R_N \to T_N/I_N$ and $R'_N \to T'_N/I'_N$. Denoting by $\ov{I}_N$ and $\ov{I}'_N$ the images of $I_N$ and $I'_N$ respectively in $\ov{T}_N$, we get maps $R_N/\varpi \to \ov{T}_N/(\ov{I}_N + \ov{I}'_N)$ and $R'_N/\varpi \to \ov{T}_N/(\ov{I}_N + \ov{I}'_N)$ which are compatible with the identification $R_N/\varpi \cong R'_N/\varpi$. 

The objects constructed above satisfy the setup described in \cite[\S 6.4.1]{10author}. Thus we can apply the results of \cite[\S 6.4.2]{10author} as in the second part of the proof of \cite[Theorem~6.4.4]{10author} to conclude that $H^*(C_0)$ is a nearly faithful $R_{\mathcal{S}_1}$-module, which implies \cref{thm:nearly_faithful}.

\bibliographystyle{amsalpha}
\bibliography{refs}

\end{document}